\documentclass[12pt,final,hidelinks,onefignum,onetabnum]{siamart250211}

\usepackage{lipsum}
\usepackage{amsfonts}
\usepackage{graphicx}
\usepackage{epstopdf}
\usepackage{algorithmic}
\ifpdf
  \DeclareGraphicsExtensions{.eps,.pdf,.png,.jpg}
\else
  \DeclareGraphicsExtensions{.eps}
\fi

\newsiamremark{remark}{Remark}
\newsiamremark{hypothesis}{Hypothesis}
\crefname{hypothesis}{Hypothesis}{Hypotheses}
\newsiamthm{claim}{Claim}
\newsiamremark{fact}{Fact}
\crefname{fact}{Fact}{Facts}

\title{
\href{http://orion.math.uwaterloo.ca/~hwolkowi/henry/reports/ABSTRACTS.html}{
Exact Solutions for the NP-hard Wasserstein Barycenter Problem
using a Doubly Nonnegative Relaxation and a Splitting Method
}
}

\author{
\href{https://uwaterloo.ca/combinatorics-and-optimization/about/people/group/50}{%
Woosuk L. Jung}\thanks{
\href{http://www.math.uwaterloo.ca/co/}{Department of Combinatorics and Optimization}, University of Waterloo,  ON, Canada (\email{w2jung@uwaterloo.ca},  {hwolkowicz@uwaterloo.ca}).}
 \and
\href{https://www.math.uwaterloo.ca/~hwolkowi/}
{Henry Wolkowicz}\footnotemark[2]
}

\usepackage{amsopn}
\DeclareMathOperator{\diag}{diag}

\usepackage{xcolor}
\definecolor{OliveGreen}{rgb}{0,0.6,0}
\definecolor{tempblue}{RGB}{36, 56, 231 }
\usepackage{cite}
\usepackage{subfigure}
\usepackage{ color, graphpap}   
\usepackage{amsmath}
\usepackage{mathabx}
\usepackage{ifthen}
\usepackage{adjustbox}

\usepackage{matlab-prettifier}

\usepackage{mathrsfs}
\usepackage{mathtools}
\usepackage[shortlabels]{enumitem}

\usepackage{physics}
\usepackage{pstricks}
\usepackage{lineno}
\usepackage{float}
\usepackage{fancyhdr}
\usepackage[us,12hr]{datetime} 
\usepackage{exscale}
\usepackage{tabularx}
\usepackage{longtable}
\usepackage{latexsym}
\usepackage{fullpage}       
\usepackage{float}
\usepackage{verbatim}
\usepackage{multirow}
\usepackage{overpic}
\usepackage{comment}

\usepackage{siunitx}
\usepackage{colortbl}

\definecolor{tablegray}{RGB}{215, 219, 221 }

\usepackage{booktabs}
\usepackage{makeidx}
 \makeindex

 \makeindex

\usepackage{csvsimple}

\hypersetup{
	plainpages=false,       
	unicode=false,          
	pdftoolbar=true,        
	pdfmenubar=true,        
	pdffitwindow=false,     
	pdfstartview={FitH},    
pdftitle={Exact Solutions Wasserstein Barycenters;
Doubly Nonnegative Relaxation; Splitting Method
	           },    
	pdfnewwindow=true,      
	colorlinks=true,        
	linkcolor=blue,         
	citecolor=green,        
	filecolor=magenta,      
	urlcolor=cyan           
}

\numberwithin{equation}{section}  
\numberwithin{table}{section}
\numberwithin{figure}{section}

\usepackage{thmtools}

\usepackage{mathtools}

\usepackage{refcount} 

\usepackage{tikz}
\usepackage{pgfplots}
\pgfplotsset{compat=1.18}
\usetikzlibrary{shapes.geometric, shapes.misc}

\def\R{\mathbb{R}}
\def\Z{\mathbb{Z}}

\def\Zp{\Z_+}

\def\Sc{\mathbb{S}}

\def\Sn{\Sc^n}

\def\Snp{\Sc_+^n}

\def\Sp{\Sc_+}

\def\Snpp{\Sc_{++}^n}

\def\Rd{\mathbb{R}^d}
\def\Rn{\mathbb{R}^n}

\def\Rnp{\mathbb{R}_+^n}
\def\Rnpp{\mathbb{R}_{++}^n}

\def\normF#1{\|#1\|_F}

\def\PSD{\mbox{\boldmath$PSD$}\,}

\def\Snplusone{\mathbb{S}^{n+1}}
\def\Sknplusone{\mathbb{S}^{kn+1}}

\def\Rnplusone{\mathbb{R}^{n+1}}

\def\Skno{\Sc^{kn+1}}
\def\Sknop{\Sc_+^{kn+1}}

\usepackage{algorithm}   

\newtheorem{example}[theorem]{Example}

\newtheorem{prop}[theorem]{Proposition}

\newtheorem{lem}[theorem]{Lemma}

\newtheorem{problem}[theorem]{Problem}

\newtheorem{cor}[theorem]{Corollary}

\crefname{thm}{Theorem}{Theorems}
\Crefname{thm}{Theorem}{Theorems}
\crefname{problem}{Problem}{Theorems}
\Crefname{problem}{Problem}{Theorems}
\Crefname{assump}{Assumption}{Theorems}
\crefname{assump}{Assumption}{Theorems}
\Crefname{linesrch}{LineSearch}{Theorems}
\crefname{linesrch}{LineSearch}{Theorems}
\crefname{conjecture}{Conjecture}{Theorems}
\Crefname{conjecture}{Conjecture}{Theorems}
\crefname{prop}{Proposition}{Propositions}
\Crefname{prop}{Proposition}{Propositions}
\crefname{cor}{Corollary}{Corollaries}
\Crefname{cor}{Corollary}{Corollaries}
\crefname{lem}{Lemma}{Lemmas}
\Crefname{lem}{Lemma}{Lemmas}
\crefname{lemma}{Lemma}{Lemmas}
\Crefname{lemma}{Lemma}{Lemmas}
\crefname{defn}{definition}{definitions}
\Crefname{defn}{Definition}{Definitions}
\crefname{conj}{Conjecture}{Conjectures}
\Crefname{conj}{Conjecture}{Conjectures}
\crefname{remark}{Remark}{Remarks}
\Crefname{remark}{Remark}{Remarks}
\crefname{rmk}{Remark}{Remarks}
\Crefname{rmk}{Remark}{Remarks}
\crefname{example}{Example}{Examples}
\Crefname{example}{Example}{Examples}
\crefname{algorithm}{Algorithm}{Algorithms}
\Crefname{algorithm}{Algorithm}{Algorithms}
\crefname{align}{}{}
\Crefname{align}{}{}
\crefname{equation}{}{}
\Crefname{equation}{}{}

\newcommand{\floor}[1]{\lfloor#1\rfloor}
\newcommand{\simplex}[1]{\Delta_{#1}}
\newcommand{\fancy}[1]{\mathcal{#1}}
\newcommand{\summ}[3]{\sum_{#1}^{#2}{#3}}
\newcommand{\BCQP}{\textbf{BCQP}\,}
\newcommand{\BCQPp}{\textbf{BCQP}}
\newcommand{\iprod}[2]{\langle#1, #2\rangle}
\newcommand{\NP}{\textbf{NP}\,}

\newcommand{\matrixOneTwo}[2]{\begin{bmatrix}
						#1 &#2
					  \end{bmatrix}}
\newcommand{\matrixThreeFour}[9]{
	\newcommand{\matrixThreeFourf}[3]{
		\begin{bmatrix}
			#1 & #2 & #3 & #4\\
			#5 & #6 & #7 & #8\\
			#9 & ##1 & ##2 & ##3
		 \end{bmatrix}
			}
	\matrixThreeFourf
					}

\newcommand{\matrixTwoOne}[2]{\begin{bmatrix}
						#1\\
						#2
					  \end{bmatrix}}
\newcommand{\matrixB}[4]{\begin{bmatrix}
						#1 & #2\\
						#3 & #4
					  \end{bmatrix}}
\newcommand{\matrixC}[9]{\begin{bmatrix}
						#1 & #2 & #3\\
						#4 & #5 & #6\\
						#7 & #8 & #9
					  \end{bmatrix}}

\newcommand{\textdef}[1]{\textit{#1}\index{#1}}

\newcommand{\KK}{{\mathcal K} }
\newcommand{\cK}{{\mathcal K} }

\newcommand{\cZ}{{\mathcal Z} }

\newcommand{\LL}{{\mathcal L} }
\newcommand{\RR}{{\mathcal R} }
\newcommand{\YY}{{\mathcal Y} }
\newcommand{\cY}{{\mathcal Y} }

\newcommand{\GG}{{\mathcal G} }

\newcommand{\cJ}{{\mathcal J} }
\newcommand{\cO}{{\mathcal O} }
\newcommand{\PP}{{\mathcal P} }

\newcommand{\Nn}{{\NN}^n}

\newcommand{\SC}{\mathcal{S}^n_C}
\newcommand{\Sh}{{\mathcal S}^n_H}

\newcommand{\cL}{{\mathcal L} }
\newcommand{\cG}{{\mathcal G} }

\newcommand{\sADMM}{\textbf{sADMM}\,} 
\newcommand{\sADMMp}{\textbf{sADMM}}
\newcommand{\EDM}{\textbf{EDM}\,}
\newcommand{\EDMp}{\textbf{EDM}}
\newcommand{\WBP}{\textbf{WBP}\,}
\newcommand{\WBPp}{\textbf{WBP}}
\newcommand{\ADMM}{\textbf{ADMM}\,}
\newcommand{\ADMMp}{\textbf{ADMM}}

\newcommand{\DNN}{\textbf{DNN}\,}
\newcommand{\DNNp}{\textbf{DNN}}
\newcommand{\KKT}{\textbf{KKT}\,}

\def\SDP{\mbox{SDP}\,}
\def\SDPp{\mbox{SDP}}

\newcommand{\LP}{\textbf{LP}\,}

\newcommand{\FR}{\textbf{FR}\,}
\newcommand{\FRp}{\textbf{FR}}

\newcommand{\cR}{{\mathcal R}}

\newcommand{\NN}{{\mathcal{N}}}

\newcommand{\A}{{\mathcal A}}

\newcommand{\bbm}{\begin{bmatrix}}
\newcommand{\ebm}{\end{bmatrix}}
\newcommand{\bem}{\begin{pmatrix}}
\newcommand{\eem}{\end{pmatrix}}
\newcommand{\beq}{\begin{equation}}
\newcommand{\beqs}{\begin{equation*}}
\newcommand{\bet}{\begin{table}}
\newcommand{\eeq}{\end{equation}}
\newcommand{\eeqs}{\end{equation*}}
\newcommand{\beqr}{\begin{eqnarray}}

\DeclareMathOperator{\BlkDiag}{{BlkDiag}}
\DeclareMathOperator{\arrowz}{{arrow_0}}

\DeclareMathOperator{\nul}{null}
\DeclareMathOperator{\range}{range}

\DeclareMathOperator{\blkdiag}{{blkdiag}}

\DeclareMathOperator{\Diag}{{Diag}}

\newcommand{\nc}{\newcommand}
\nc{\arrow}{{\rm arrow\,}}
\nc{\Arrow}{{\rm Arrow\,}}
\nc{\BoDiag}{{\rm B^0Diag\,}}
\nc{\bodiag}{{\rm b^0diag\,}}

\nc{\Mm}{{\mathcal M}^{m} }
\nc{\Mmn}{{\mathcal M}^{mn} }
\nc{\Mnr}{{\mathcal M}_{nr} }
\nc{\Mnmr}{{\mathcal M}_{(n-1)r} }
\nc{\kwqqp}{Q{$^2$}P\,}
\nc{\kwqqps}{Q{$^2$}Ps}

\nc{\notinaho}{(X,S)\in \overline{AHO}(\A)}
\nc{\inaho}{(X,S)\in AHO(\A)}

\newcommand{\bea}{\begin{eqnarray}}%
\newcommand{\eea}{\end{eqnarray}}%
\newcommand{\beas}{\begin{eqnarray*}}%
\newcommand{\eeas}{\end{eqnarray*}}%
\newcommand{\Int}{{\rm int\,}}

{}

\newcommand{\Hnp}[1][]{\,\mathbb{H}_+^{\ifthenelse{\equal{#1}{}}{n}{#1}}}
\newcommand{\Hn}[1][]{\,\mathbb{H}^{\ifthenelse{\equal{#1}{}}{n}{#1}}}
\newcommand{\Hk}[1][]{\,\mathbb{H}^{\ifthenelse{\equal{#1}{}}{k}{#1}}}
\newcommand{\Dn}[1][]{\,\mathbb{D}^{\ifthenelse{\equal{#1}{}}{n}{#1}}}

\newcommand{\eig}{\text{eig}}

\DeclareMathOperator*{\argmin}{argmin}

\title{
\href{http://orion.math.uwaterloo.ca/~hwolkowi/henry/reports/ABSTRACTS.html}{
Exact Solutions for the NP-hard Wasserstein Barycenter Problem
using a Doubly Nonnegative Relaxation and a Splitting Method}
   \footnote{
Emails resp.:  w2jung@uwaterloo.ca, hwolkowicz@uwaterloo.ca}
}

\author{
\href{https://uwaterloo.ca/combinatorics-and-optimization/about/people/group/50}{%
Woosuk L. Jung}\thanks{
\href{http://www.math.uwaterloo.ca/co/}{Department of Combinatorics and Optimization}, University of Waterloo,  ON, Canada}
 \and
\href{https://www.math.uwaterloo.ca/~hwolkowi/}
{Henry Wolkowicz}\footnotemark[3]
}

\date{
Revising as of \today, \currenttime
}

\begin{document}
\maketitle
\tableofcontents
\nolinenumbers



%
%

\begin{abstract}

The so-called
\emph{simplified} Wasserstein barycenter problem, also known as the
cheapest hub problem, consists in selecting one
point from each of $k$  given sets, each set consisting of $n$ points, 
with the aim of
minimizing  the sum of distances to the barycenter of the $k$ chosen points.
This problem is known to be NP-hard. 
We compute the Wasserstein barycenter by exploiting the
Euclidean distance matrix structure to obtain a facially reduced doubly
nonnegative, \DNNp, relaxation. The facial reduction provides a
natural splitting for applying the symmetric alternating directions 
method of multipliers (\sADMM) to the \DNN relaxation. The \sADMM
method exploits structure in the subproblems to find
strong upper and lower bounds.
In addition, we extend the problem to allow varying $n_j$ points for
the $j$-th set.

The purpose of this paper is twofold. First we want
to illustrate the strength of this \DNN relaxation with the natural 
splitting approach mentioned above. 
Our numerical tests then illustrate the
surprising success on random problems, as we generally,
efficiently, find the provable exact solution of this NP-hard problem.
Comparisons with current commercial software illustrate this surprising
efficiency.  However, we demonstrate and prove that there is a duality gap 
for problems with \emph{enough} multiple optimal solutions, and that
this arises from problems with highly symmetrized structure.
\end{abstract}

\section{Introduction}
\label{sect:intro}
\index{\WBPp, simplified Wasserstein barycenter problem}

We consider the so-called \textdef{simplified Wasserstein barycenter,
\WBPp} problem 
of finding the optimal barycenter of $k$ points, where
exactly one point is chosen from $k$ sets of points, each
set consisting of $n$ points. This is a simplification of more general
problems of \textdef{optimal mass transportation} and the problems of
summarizing and combining probability measures that occurs
in  e.g.,~statistics and machine learning.  
In \cite[Def. 1.4]{MR4378594} this problem is called the
\textdef{cheapest-hub problem},   and further in \cite{MR4378594},
a reduction to \WBP is derived from the \textdef{$k$-clique problem} thus 
proving NP-hardness.\footnote{Recall that the 
\textdef{$k$-clique problem} is the problem of
finding $k$ vertices in a graph such  that each pair is close in
the sense of being \emph{adjacent}.}
Algorithms for \WBP with exponential dependence in $d$ are discussed
in~\cite[Sect. 1.3.1]{MR4378594}.\footnote{We discuss this further below
as the complexity of our algorithm does \emph{not} depend on $d$.}
There are many important applications in molecular
conformation e.g.,~\cite{BurkImWolk:20}, clustering~ \cite{docuC2},
supervised and unsupervised learning, etc.
For additional details on the theory and applications of optimal
transport theory see 
e.g.,~\cite{altschuler2020wasserstein,Panaretos_2019,NEURIPS2019Vanc},
\href{https://www.youtube.com/watch?v=k1CeOJdQQrc}{lecture link},
\href{https://alexhwilliams.info/itsneuronalblog/2020/10/09/optimal-transport/}{Introduction
to Transportation Problems link},
and the many references therein.

The purpose of this paper is twofold. First, we provide a successful
framework for handling quadratic hard discrete optimization problems; 
and second, we illustrate the surprising success when applied to
our specific \WBPp. 

We model our problem as a quadratic objective, quadratic constrained
$\{0,1\}$ discrete optimization problem, i.e.,~we
obtain a \emph{binary quadratic} model.
We then get a convex relaxation for
this hard discrete optimization problem by lifting
to the doubly nonnegative, \DNNp, cone, the cone of nonnegative elementwise,
positive semidefinite symmetric matrices.
Strict feasibility fails for the
relaxation, so we apply \textdef{facial reduction, \FRp}.
This results in many constraints becoming redundant and also
gives rise to a \emph{natural splitting} that can be exploited by
the symmetric alternating directions method of multipliers (\sADMM).
We exploit the structure, and include redundant constraints on the
subproblems of the splitting and on the dual variables. The \sADMM
algorithm allows for efficient upper and provable lower bounding techniques 
for the original hard \WBP problem. This helps the algorithm stop early.

Our Extensive tests on random problems are
surprisingly efficient and successful, 
i.e.,~the relaxation with the upper and lower bounding techniques
provide a provable optimal solution to the original hard \WBP for
\emph{surprisingly many} instances; essentially for all our randomly
generated instances. 
For example, to solve the original hard discrete optimization problem 
to optimality for our algorithm for a random
problem with $k=n=25$ in dimension $d=25$ took of the order of $10$
seconds. In contrast, using the well known solver
\href{https://www.gurobi.com}{\textsc{Gurobi}}, on three problems, 
with sizes $n = k =5,7,8$, respectively,  it
took approximately:
$2, 570, 46692$ seconds, respectively (using the same laptop for both
tests).
(Detailed numerics for our algorithm are provided below
\Cref{sect:numerics}.)

\label{dualitygap}
Though generically we are surprisingly successful in finding the exact
solutions to the NP-hard problem,
the \DNN relaxation can fail to find the exact solution
for problems with special structure, i.e.,~there can be a 
positive duality gap between the optimal value of the original hard problem and
the lower bound found from the \DNN relaxation.
In~\Cref{cor:dualgapunc} we include a constructive proof for a general
hard $\{0,1\}$ problem 
that a sufficient number of linearly independent
optimal solutions results in a duality gap between the original hard
problem and its \DNN relaxation. 
This is specialized to our NP-hard problem in~\Cref{cor:dualgapWass}.
A specific instance is included.
Note that we consider that we have an optimal solution to \WBP if the
upper and lower bounds are equal to machine precision as any other
feasible solution cannot have a smaller objective 
value within machine precision.

\subsection{Outline}
We continue in \Cref{sect:notation} with preliminary notation.
The main NP-hard~\Cref{prob:mainWB} and connections to Euclidean
distance matrices, \EDMp, are given in~\Cref{sect:simpWB}.
A regularized, facially reduced, doubly nonnegative, \DNNp, 
relaxation is derived
in~\Cref{sect:relaxprob}. 
The \FR in the relaxation fits \emph{naturally} with applying a
splitting approach. This is presented in~\Cref{sect:ADMMalg} along with
special bounding techniques and heuristics 
on the dual multipliers for accelerating the splitting algorithm.
The algorithm provides \emph{provable} 
lower and upper bounds for the
original NP-hard~\Cref{prob:mainWB} that are attained by dual and primal
feasible solutions, respectively. Thus a \emph{zero} gap
(called a duality gap) proves optimality.
Our empirics are given in~\Cref{sect:numerics}.

In~\Cref{sect:multoptgaps} we prove that multiple optimal solutions 
can lead to duality gaps. We include specific instances.
Our concluding remarks are in~\Cref{sect:concl}.

\subsection{Notation}
\label{sect:notation}
We let \textdef{$S \in \Sn$} denote a matrix in the space of 
$n\times n$ symmetric matrices equipped with the \textdef{trace inner
product} $\langle S,T\rangle = \trace ST$; 
we use \textdef{$\diag(S)\in \Rn$} to denote the linear mapping to
the diagonal of $S$; the adjoint mapping is
\textdef{$\diag^*(v)=\Diag(v)\in \Sn$}.
We let \textdef{$[k] = \{1,2,\ldots,k\}$}.

The convex cone of positive semidefinite matrices is denoted $\Snp \subset
\Sn$, and we use $X\succeq 0$ for $X\in \Snp$. Similarly, for positive
definite matrices we use $\Snpp, X\succ 0$. We let $\Nn$ denote $n\times
n$ nonnegative symmetric matrices.  The cone of doubly nonnegative matrices is 
$\DNN = \Snp \cap \Nn$.
\index{\DNNp, doubly nonnegative}
\index{doubly nonnegative, \DNNp}

For a set of points $p_i\in \R^d$, we let $P=\begin{bmatrix} p_1 &
p_2 & \ldots & p_t \end{bmatrix}^T\in \R^{t\times d}$ denote
the \textdef{configuration matrix}. Here $d$ is the
\textdef{embedding dimension}. Without loss of
generality, we can  assume the points span $\Rd$, and we can
translate the points and assume they are centered,
i.e.,
\index{vector of ones, $e$}
\index{$0$-th unit vector, $e_0$}
\[
P^Te_t = 0.\footnote{
	The translation is given by
	\[
	P^T \mapsto P^T-ve^T,
	\]
	where $v:= \frac{1}{n}P^Te$ is the barycenter of the points.}
\]
Here we let \textdef{$e_t$, vector of ones} of dimension $t$ and we use
$e$ if the dimension is clear. And we use \textdef{$e_0$, $0$-th unit
vector} of appropriate size.  We denote the corresponding
\textdef{Gram matrix, $G=PP^T$}. Then the classical result of
Schoenberg~\cite[Sect. 3]{MR1501980}, see e.g.~\cite{alfm18}, relates a
\textdef{Euclidean distance matrix, \EDMp}, with a Gram matrix by applying
\[
\textdef{$D = \KK(G) = \diag(G)e^T+e\diag(G)^T-2G$}.
\]
Moreover, this mapping is one-one and onto between the $t(n-1)$
dimensinal
\textdef{centered subspace, $\SC$} and \textdef{hollow subspace, $\Sh$}
\index{$\SC$, centered} 
\index{$\Sh$, hollow} 
\[
\SC = \{X\in \Sn \,:\, Xe = 0\}, \quad
\Sh = \{X\in \Sn \,:\, \diag X = 0\}; \quad \cK(\SC) = \Sh.
\]
Here we denote the \textdef{triangular number, $t(k) = k(k+1)/2$}.
Note that the centered assumption $P^Te=0 \iff G=PP^T\in \SC$.
Further notation is introduced as needed.

\index{\EDMp, Euclidean distance matrix}
	
\section{Simplified Wasserstein barycenters, \WBPp}
\label{sect:simpWB}
We now present the main problem and the connections to Euclidean distance
matrices, \EDMp. We follow the notation in \cite[Sect. 1.2]{MR4378594}
and refer to our main problem as the simplified Wasserstein barycenter
problem, or Wasserstein barycenter for short.

\subsection{Main problem and \EDM connection}
Our main optimization problem is to find a point in each of $k$ sets
to obtain an optimal barycenter. We can think of this as finding a
\textdef{hub of hubs}. That is, suppose that there are $k$ areas with $n$
airports in each area.\footnote{We extend the prolem to allow for
different sizes for the sets.} We want to choose exactly one airport to act
as a minor hub in each of
the $k$ areas so that the barycenter for these $k$ minor hubs would
serve as a major (best) hub for the $k$ minor hubs. 
\index{cheapest-hub problem}.
\begin{problem}[{\textdef{simplified Wasserstein barycenter, \WBPp}}]
\label{prob:mainWB}
	Suppose that we are given a finite number of sets $S_1,...,S_k$, each 
	consisting of $n$ points in $\R^d$. Find the optimal barycenter point
	$y$ after choosing exactly one point from each set:
	\begin{equation}
		\label{eq:mainprob}
		\textdef{$p^*_W$} :=
		\min_{\stackrel{p_i\in S_i}{i\in[k]}} \min_{y\in\R^d}\summ{i\in [k]}{}{\norm{p_i-y}{}^2}
		=:\min_{\stackrel{p_{j_i}\in S_i}{i\in[k]}}
		\textdef{$F(p_{j_1},p_{j_2},\ldots,p_{j_k})$},
	\end{equation}
thus defining $F$. Here
\index{$P$, configuration matrix}
\index{configuration matrix, $P$}
\begin{equation}
	\label{e:defPDG}
P^T=\begin{bmatrix} p_1 & \ldots & p_N \end{bmatrix}\in \R^{d\times N}, D,G,
\end{equation}
 denote the 
corresponding (configuration) matrix of points, \EDMp, and Gram matrix,
respectively. We allow the set sizes, \textdef{$n_j, j\in [k]$}, 
to vary and let \textdef{$N=\sum_{j\in[k]} n_j$}.
\end{problem}
By~\Cref{lem:barycequival} below, the optimal Wasserstein barycenter is
the standard barycenter of the $k$ optimal points.
It is known \cite[Sect. 1.2]{MR4378594} that the problem can be 
phrased using inter-point squared distances. We include a proof to
emphasize the connection between Gram and Euclidean distance
matrices.\footnote{As noted earlier, This is called 
the \emph{cheapest-hub} problem in \cite[Sect. 1.2]{MR4378594}.}
We start by recording the following minimal property of the standard
barycenter with respect to sum of squared distances.
\begin{lemma}
\label{lem:barycequival}
Suppose that we are given $k$ points $q_i\in \Rd, i=1,\ldots k$. Let
$\bar y=\frac 1k\sum_{i=1}^k q_i$ denote the barycenter. Then
\[
\bar y = \argmin_y \sum_{i=1}^k \frac 12\|q_i-y\|^2.
\]
\end{lemma}
\begin{proof}
The result follows from the stationary point equation
$\sum_{i=1}^k (q_i -\bar y) = 0$.
\end{proof}
We  now have the following useful lemma.
\begin{lem}
Let $Q^T=[q_1 \ldots q_k]\in \R^{d\times k}$ and let $G_Q$ and $D_Q$ be, respectively, the Gram 
	and the EDM matrices corresponding to the columns in $Q^T$. Further, 
	let $y = \frac 1k Q^Te$ be the barycenter. Then 
	\begin{equation}
		\label{eq:GDgenerali}
	2k \sum_{i=1}^k ||q_i - y||^2 = 
e^T D_Q e = 2 k \trace (G_Q) - 2 e^TG_Qe.
   \end{equation}
\end{lem}

\begin{proof}
Let $J= I - ee^T/k$ be the
orthogonal projection onto $e^\perp$. Hence, $J^2=J^T=J$. Moreover, the
$i$-th row $(JQ)_i = (Q-\frac 1k ee^TQ)_i=(q_i-y)^T$. Now   
	\[
	\sum_{i=1}^k ||q_i - y||^2  = \trace  ( J Q Q^T J ) = \trace  (
J G_Q ) 
	= \trace (G_Q) - \frac{1}{k} e^TG_Qe. 
	\] 
	But $D_Q = \cK (G_Q) = e \diag(G_Q)^T + \diag(G_Q) e^T - 2 G_Q$.
Therefore, $e^T D_Q e = 2 k \trace (G_Q) - 2 e^TG_Qe$.
\end{proof}

The following~\Cref{prop:equivprobs} illustrates the connection between
the simple Wasserstein barycenter problem\footnote{We refer to this as the
Wasserstein barycenter problem.}
of finding the optimal barycenter and the
$k$-clique problem of finding $k$ pairwise adjacent vertices.
\begin{corollary}
\label{prop:equivprobs}
Let $N=\sum_{j\in[k]} n_j$, and let 
$ S_j=\{p_{j_i}\}_{i=1}^{n_j}, j\in [k]$ be given.
Consider the main problem~\cref{eq:mainprob} and let $y$ denote the 
optimal Wasserstein barycenter. The problem of finding $y$
 is equivalent to finding 
exactly one point in each set that minimizes the sum of squared distances
in  the following:
	\begin{equation}
		\label{eq:mainprobl}
		(\textdef{WIQP}) \qquad
2Np_W^*=p^*:=
\min_{\stackrel{p_t \in S_t}{t\in [k]}}\summ{i,j\in [k]}{}{\norm{p_i-p_j}{}^2}.
\end{equation}
\end{corollary}
\index{$p^*=2kp_W^*$}
\begin{proof}
Suppose that $P^T=\begin{bmatrix} p_1 & \ldots & p_k \end{bmatrix}\in
\R^{d\times k}$ is a
fixed matrix of optimal solution vectors to  \cref{eq:mainprob}, 
and let $y$ be the
barycenter. Without loss of generality, since distances do not 
change after a translation, we translate all the
points $p_t, t\in[k]$, by $y$ and obtain $y=0$. 
This implies that the corresponding Gram matrix  $Ge=PP^Te=0$.
This combined with \cref{eq:mainprob} and \cref{eq:GDgenerali} and the
corresponding distance matrix $D$ yield 
\begin{equation}
\label{eq:eDe}
\begin{array}{rcl}
\summ{i,j\in [k]}{}{\norm{p_i-p_j}{}^2}
&=& 
e^TDe 
\\&=& 
e^T(\diag(G)e^T+e\diag(G)^T-2G)e
\\&=& 
2N\diag(G)^Te
\\&=& 
2N\trace G
\\&=& 
2N\sum_{i\in[k]} \|p_i\|^2
\\&=& 2Np_W^*,
\end{array}
\end{equation}
where the last equality follows from~\Cref{lem:barycequival}. 
\end{proof}

\subsection{A reformulation using a Euclidean distance matrix}
In this paper we work with the optimal value
$p^*$ and now provide a reformulation of
\cref{eq:mainprobl} using an \EDMp.
Define 
\begin{equation}
\label{eq:xAdef}
x:=\begin{pmatrix} v_1 \cr \vdots \cr v_k\end{pmatrix}\in \R^{N}, 
\, v_i \in \R^{n_i},
\quad 
\textdef{$A:=\blkdiag(e_{n_1}^T,...,e_{n_k}^T)\in \R^{k\times N}$}.
\index{Kronecker product, $\otimes$}
\index{$\otimes$, Kronecker product}
\end{equation}
And we note that $A= I\otimes e^T \in\R^{k\times kn}$, if $n_i=n, \forall i$,
where we denote \textdef{Kronecker product, $\otimes$}.
Note that we get $A^Te = e$.
	Then, the constraints of picking exactly one point from each set
can be recast as:
\begin{equation}
\label{eq:1each}
Ax=e, \, x \in \{0,1\}^N.
\end{equation} 
	
Recalling \Cref{prop:equivprobs} and \cref{eq:eDe} in the proof,
we see that  \cref{eq:mainprob} can be formulated as  a
binary-constrained quadratic program (\textdef{\BCQP}) using the Euclidean
distance matrix $D$ formed from all $N$ points $P^T=\begin{bmatrix}
p_{1} & \ldots & p_{n_1} &
p_{n_1+1} & \ldots & p_{n_1+n_2} &
\ldots & 
& \ldots & p_{N}
\end{bmatrix}$:
\begin{equation}
\label{eq:BCQP}
(\BCQP)\qquad \begin{array}{rcl}
\textdef{$p^*$} = 
&   \min & x^TDx=\iprod{D}{xx^T} \\
& \text{s.t.} & Ax=e   \\
&&	x\in \{0,1\}^{N}.
\end{array}
\end{equation}
For simplicity in the sequel we often assume that the cardinality of all
sets are equal. 

\begin{remark}[difficulty of the Wasserstein barycenter problem]
We first note that $A$ in \cref{eq:xAdef} is \textdef{totally unimodular}, 
i.e.,~every square submatrix has $\det(A_I)\in \{0,\pm 1\}$. Therefore, 
the basic feasible solutions, \textdef{vertices of the feasible set}, of
the LP relaxation $Ax=e, \, x\geq 0$, are $\{0,1\}$ variables.
Therefore, these discrete optimization problems with a 
linear objective yield vertices as  optimal
solutions and can be solved with simplex type methods while yielding
$\{0,1\}$ solutions.

For our problem we have a quadratic objective function.
And, by the properties of distance matrices,
it  is concave on the span of the feasible set of the LP relaxation.
Therefore, minima are attained at extreme points, i.e.,~at vertices,
at $\{0,1\}$ points. But solving the
concave minimization problem is a hard problem.
Note that the \emph{width} of $P$ is hidden 
in the rank when forming the Gram matrix $G=PP^T$ and so hidden in the
rank of the rank-two update \EDM
\[
D= \diag(G)e^T+e\diag(G)^T-2G, \quad \rank(D)\in 
[\rank(G),\rank(G)+2].
\]

In summary, the problem appears to be NP-hard due to the minimization of
a \underline{quadratic} function, \cite{Pard:91}, and the binary 
$0,1$ constraints. This is in contrast to the linear programming approaches
for the generalized transportation problems solved in
e.g.,~\cite{NEURIPS2019Vanc} and the references therein.
However, the properties of total unimodularity and concavity both
promote binary valued optimal points.
\end{remark}

\section{Facially reduced \DNN relaxation}
\label{sect:relaxprob}
We now introduce a regularized convex relaxation to the hard binary quadratic
constrained problem introduced in~\cref{eq:BCQP}.
We start with the \SDP relaxation using a standard lifting approach.
We regularize using \textdef{facial reduction, \FRp}, and
include the so-called \emph{gangster constraint}.
We then strengthen this by including nonnegativity
constraints, i.e.,~we get the \textdef{doubly nonnegative, \DNNp}
relaxation. 

\index{\FRp, facial reduction}
\index{\DNNp, doubly nonnegative}

As stated above, for simplicity of notation we consider all sets to have
the same cardinality, $n_1=\ldots = n_k, N = \sum_j n_j$.

\subsection{Semidefinite programming (\SDPp) relaxation}
We begin with deriving an \SDP relaxation of our formulation in
\cref{eq:BCQP}. We start with a feasible vector $x\in \R^{kn}$ and set
$\begin{pmatrix} x_0 \cr x\end{pmatrix} =
\begin{pmatrix} 1 \cr x\end{pmatrix}$. We then
lift the vector to a rank-$1$ matrix 
$Y_x:= \begin{pmatrix} 1 \cr x\end{pmatrix}
\begin{pmatrix} 1 \cr x\end{pmatrix}^T$. 
The convex hull of the lifted vertices of the feasible set
of~\cref{eq:BCQP} yields an equivalent polyhedral set in $\Sc^{N+1}$.
It is difficult and expensive to find this polyhedral set.
To obtain a tractable convex relaxation,
we relax the implicit nonconvex rank-$1$ constraint on $Y_x$ and linearize 
the objective function. Let
\begin{equation}
\hat{D}:=\matrixB{0}{0}{0}{D}\in  \Sc^{kn+1}.
\end{equation}
The objective function of \cref{eq:BCQP} now becomes $\iprod{D}{xx^T}=\iprod{\hat{D}}{Y_x}$.
After the lifting, we impose the constraints that we have
from $x$ onto $Y$, e.g.,~these include
the $\{0,1\}$-constraints $x_i^2-x_i=0$, and the linear constraints $Ax=e$.
\index{zero-th unit vector, $e_0$}

\subsubsection{\SDP reformulation}
Define the linear transformation
\begin{equation*}
\textdef{$\arrow$}:\Snplusone\rightarrow\Rnplusone: \matrixB{s_0}{s^T}{s}{\bar{S}}\mapsto
\begin{pmatrix} s_0\cr \diag(\bar{S})-s\end{pmatrix}.
\end{equation*}
For convenience, we define
\begin{equation*}
\textdef{$\arrowz$}:\Snplusone\rightarrow\Rnplusone: \matrixB{s_0}{s^T}{s}{\bar{S}}\mapsto
\begin{pmatrix} 0\cr \diag(\bar{S})-s\end{pmatrix}.
\end{equation*}	

We now show that, as long as the rank-one condition holds, 
the binary constraint on vector $x$ is
equivalent to the $\arrow$ constraint on the lifted matrix, $\arrow(Y_x)=e_0$. 
\begin{prop} The following holds:
\[
\left\{Y\in\Skno_+: \rank(Y)=1, \arrow(Y)=e_0\right\} = \left\{Y=
\begin{pmatrix} 1\cr x\end{pmatrix}
\begin{pmatrix} 1\cr x\end{pmatrix}^T
: x\in\{0,1\}^{kn}\right\}.\]
\end{prop}
\begin{proof}
$(\supseteq)$: This is clear from the definitions.\\
$(\subseteq)$: Since $Y$ is symmetric, positive semidefinite and has
rank $1$, there exist  $x_0\in \R$ and $x\in \R^{kn}$
such that $Y =
\begin{pmatrix} x_0\cr x\end{pmatrix}
\begin{pmatrix} x_0\cr x \end{pmatrix}^T$.
Since $\arrow(Y)=e_0$,
$x_0^2=1$ and $x\circ x = x_0x$. If $x_0=1, x\in\{0,1\}^{kn}$; otherwise
$x_0=-1$ and $x\in\{0,-1\}^n$ and it is easy to verify that 
\[\left\{
\begin{pmatrix} 1\cr x\end{pmatrix}
\begin{pmatrix} 1\cr x\end{pmatrix}^T
: x\in\{0,1\}^{kn}\right\} = \left\{
\begin{pmatrix} -1\cr x\end{pmatrix}
\begin{pmatrix} -1\cr x\end{pmatrix}^T
: x\in\{0,-1\}^n\right\}.\]
\end{proof}
For the ``only-one-element-from-each-set'' linear equality constraint 
$Ax=e$ (see \cref{eq:1each}), 
we use the following positive semidefinite matrix
\begin{equation}
\label{eq:matK}
K:=\matrixTwoOne{-e^T}{A^T}\matrixTwoOne{-e^T}{A^T}^T\in
\Sc^{kn+1}_+.
\end{equation}
We observe that
\begin{equation}
\label{eq:lineqnAK}
\begin{array}{rcl}
Ax=e &\iff & \begin{pmatrix} 1\cr x\end{pmatrix}^T
\matrixTwoOne{-e^T}{A^T}=0
\\&	\iff &   Y_xK =
\begin{pmatrix} 1\cr x\end{pmatrix}
\begin{pmatrix} 1\cr x\end{pmatrix}^T
\matrixTwoOne{-e^T}{A^T}\matrixTwoOne{-e^T}{A^T}^T=0
\\  & \iff &KY_x=0, 
\end{array}
\end{equation}
i.e.,
$\range(Y_x)\subseteq\nul(K)=\nul\left(\begin{bmatrix}-e&A\end{bmatrix}\right)$.
Moreover, this
emphasizes that strict feasibility fails for feasible $Y$ even if we
ignore the rank-1 constraint. If we choose $V$ full column rank so that $\range(V) = \nul(K)$, then we can
\emph{facially reduce} the problem using the substitution
\begin{equation}
\label{eq:facialvector}
Y \leftarrow VR  V^T  \in V\Sp^{nk+1-k}V^T \unlhd \Sc^{kn+1}_+,
\end{equation}
where $\unlhd$ denotes \emph{face of}.
This makes the constraint $KY=0$ redundant. More detailed discussion for constructing the facial vector $V$ is provided later in \Cref{sec:facialvector}.

Then the rank restricted \SDP reformulation of  \cref{eq:BCQP} becomes
\[(\SDP)\qquad \begin{array}{rcl}
\label{e:SDPr}
\textdef{$p^*$} = & \min & \iprod{\hat{D}}{Y}\\
  				&		&\arrow(Y)=e_0\\
				&		&\rank(Y) = 1\\
				&		&KY = 0 \\
				&		& Y \in \Sknop.
\end{array}
\]

\subsubsection{Relaxing the rank-$1$ constraint}
 Since the $\NP$-hardness of the \SDP formulation comes from the
rank-$1$ constraint, we now relax the problem by deleting this
constraint. The \SDP relaxation of the above model is
\begin{equation}
\label{eq:SDPrelaxone}
(\SDP \text{ relax})\qquad \begin{array}{rcl}
		  \textdef{$p^*$} = & \min_{Y\in\Skno} & \iprod{\hat{D}}{Y}\\
					&		      			  &\arrow(Y)=e_0\\
			&   			  &KY = 0 \\
			 &     & Y\succeq 0.
				\end{array}
\end{equation}
However, the improved processing efficiency of this convex
relaxation trades off with the accuracy of solving the original
NP-hard problem. The rank of an optimal $Y$ can now be greater
than one. The idea now is to impose a ``correct'' amount of
redundant constraints in the \SDP model that reduces the rank of an
optimal solution as much as possible, but does not hurt the processing
efficiency of the model too much. Note that if we have ignored the rank
one constraint and replaced the $KY=0$ constraint using facial reduction
and substituting $Y=VRV^T$, then strict feasibility holds, i.e.,~the
barycenter of the  
the lifted vertices of the feasible set
of~\cref{eq:BCQP} yields an $\hat R, Y = V\hat RV^T$, that satisfies
strict feasibility, e.g.,~\cite{DrusWolk:16}.
		 
\subsubsection{The gangster constraint}
 The \textdef{gangster constraint} is essentially a trivial projection
that fixes at $0$ (shoots holes at)
certain entries of the matrix. The entries are given in the
\textdef{gangster index, $\cJ$}. By abuse of notation, we allow one entry to be
fixed at $1$.
The gangster constraint in our case comes from the
linear constraint $Ax=e$ combined with the binary constraint on $x$.
We let $S \circ T$ denote the Hadamard (elementwise) product.
\index{$S \circ T$, Hadamard (elementwise) product}
\index{Hadamard (elementwise) product, $S \circ T$}
\index{gangster index, $\cJ$}
\begin{prop} 
\label{prop:gangstindices}
Let $x$ be feasible for \BCQPp. Then
		\[[A^TA-I]\circ xx^T = 0,\]
and $A^TA-I \geq 0, xx^T\geq 0$. 
\end{prop}
\begin{proof}
Recall that $x\in \R_+^{kn}$.
We now use basic properties of the Kronecker product,
e.g.,~\cite{schaecke:04}, and see that
\index{Kronecker product, $\otimes$}
\index{$\otimes$, Kronecker product}
\[
A =I_k \otimes e^T, \, 
A^T =I_k \otimes e, \, \quad
A^TA = I_k\otimes ee^T,
\]
i.e.,~$A^TA = \BlkDiag(ee^T,\ldots, ee^T)$, a block diagonal structure, and the columns of $A$ are unit vectors. 
Therefore $A^Te_k = e_{nk}$ and
$\Diag(\diag(A^TA)) = I_{kn}$. 
The nonnegativity results follow from the definition, as does $Y_{00}=1$. 

Then
\[
\begin{array}{rll}
		Ax=e\implies  & A^TAx = A^Te\\
			\implies  & A^TAx - Ix = A^Te - Ix\\
			 \implies & (A^TA - I)x = e_{nk}-x\\
			 \implies & (A^TA - I)xx^T = (e-x)x^T = ex^T-xx^T\\
			 \implies & \trace[(A^TA - I)xx^T] =
\trace[ex^T-xx^T] = \summ{i=1}{kn}{x_i-x_i^2}=0\\
			 \implies &(A^TA - I)\circ xx^T = 0.
	 	\end{array}\]
The final conclusion now follows from the nonnegativities in the
Hadamard product.
		\end{proof}
Define the gangster indices 
\[
\textdef{$\cJ$} :=
\left\{ (i,j) : \left( A^TA-I\right)_{ij} > 0
\right\}.
\]
The gangster constraint on $Y$ in \cref{eq:SDPrelaxone} is $Y_{00}=1$ and
\[
 \cJ(Y) = Y_\cJ = 0 \in \R^{|\cJ|}.
\]		
\index{$\cJ$, gangster indices}
From \Cref{prop:gangstindices}, we see that the 
\textdef{gangster indices, $\cJ$}, are
the nonzeros of the matrix $A^TA-I$,
i.e.,~the set of off-diagonal indices of the $n$-by-$n$ 
diagonal blocks of, all but the $0$-th row and column, of $Y_x$.
Our complete gangster index is \textdef{$\hat{\cJ}:=\{(0,0)\}\cup \cJ$}. 
We define the \textdef{gangster constraint mapping, $\cG_{\cJ}$}:
\index{$\cG_{\cJ}$, gangster constraint mapping}
\[
\cG_{\cJ}(Y) = Y(\cJ) \in \R^{|\cJ|}, 
\]
i.e.,~the elements of $Y$ indexed by the index set $\cJ$.
		
Now the \SDP relaxation model becomes
		\begin{equation}\label{eq:SDPrelax}
		\begin{array}{rcl}
		  \textdef{$p^*$} = & \min_{Y\in\Skno} & \iprod{\hat{D}}{Y}\\
	& &Y_{00} = 1 \\
	& &\arrowz(Y)=0\\
	&  &\GG_{\cJ}(Y)=0\\
				    	&
&KY = 0 \\
			 &     & Y\succeq 0.
		\end{array}
		\end{equation}
\begin{prop}
Consider the \SDP relaxation \cref{eq:SDPrelax} but without the
$\arrowz$ constraint.
Then every optimal solution satisfies the constraint
\[
\arrowz(Y) = 0,
\]
i.e.,~it was a redundant constraint.
\end{prop}
\begin{proof}
It follows from \cite[Thm 2.1]{BurkImWolk:20}.
\end{proof}
Our empirical tests on random problems
without the $\arrowz$ constraint confirmed this result. However, the
extra redundant constraint is useful for the subproblems in the
splitting approach below.

	\subsection{Doubly nonnegative (\DNNp) relaxation}
We now split the problem by using two variables 
$\{Y,R\}$ and apply a doubly nonnegative relaxation to
\cref{eq:SDPrelax}. This \emph{natural splitting} uses the facial
reduction obtained in~\cref{eq:facialvector} but with orthonormal
columns chosen for the facial vector $V$.
	
\subsubsection{Formulating facial vector $V$ for sparsity}
\label{sec:facialvector}
In this subsection, we suggest strategies for constructing 
sparse versions of the \textdef{facial vector, $V$}. Recall that the columns of 
the facial vector $V$ form an orthonormal basis for the nullspace of $K$.
\index{$V$, facial vector}

\index{$\unlhd$, face of}
\index{face of, $\unlhd$}

For a typical matrix $V$ see~\Cref{fig:Vmat} that is constructed using
\Cref{lem:findV}, below.
\begin{figure}[ht!]
\centering
\vspace{-2.2in}
\includegraphics[width=14cm]{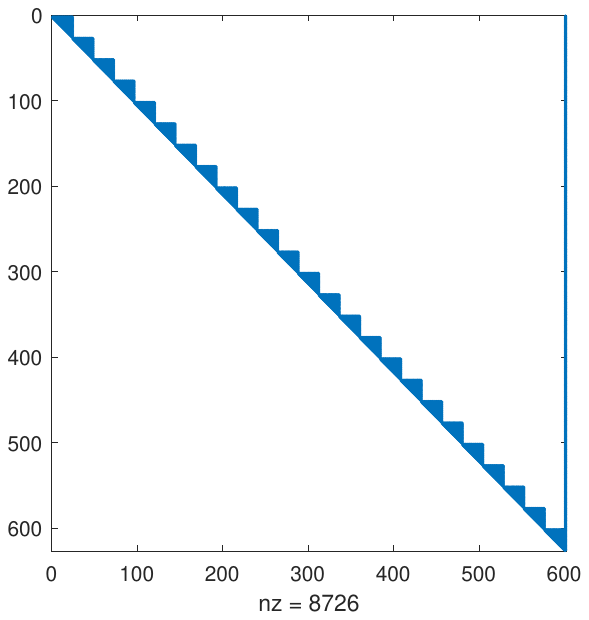}
\vspace{-2.0in}
\caption{V matrix for k=25, n=25 }
\label{fig:Vmat}
\end{figure}
Alternatively, we can use the \textsc{Matlab} 
QR algorithm (with $A$ specified as sparse) 
$[q,\sim]=qr(\begin{matrix} -e & A \end{matrix})$ 
and use the last part of $q$ for the nullspace.
This results in a relatively sparse orthonormal basis for the nullspace.

\begin{lemma}
	\label{lem:findV}
	Let $k,n$ be given positive integers and from above let
	\[
	A = \begin{bmatrix} I_k\otimes e_n^T \end{bmatrix}, \,
	B = \begin{bmatrix} -e_k & A \end{bmatrix}.
	\]
	Let $\cO \in \R^{n-1\times n-1}$ be the strictly upper triangular matrix 
	of ones of order $n-1$. Set 
	\[
	v=\begin{pmatrix} \frac 1{\sqrt{j+j^2}}\end{pmatrix}_j\in \R^{n-1}, \,
	\bar v=\begin{pmatrix} \frac j{\sqrt{j+j^2}}\end{pmatrix}_j\in \R^{n-1}, \,
	\beta = -1/\sqrt{n^2+nk},\, \text{and } \alpha =  n\beta.
	\]
	Let $\tilde \cO = -\cO \Diag(v) + \Diag(\bar v)$ and set
	\[
	\bar \cO = \begin{bmatrix}
		-v^T \\ \tilde \cO
	\end{bmatrix}
	=
	\begin{bmatrix}
		-v_1 & -v_2 & -v_3 & \cdots & -v_{n-1}\\
		\bar v_1 & -v_2 & -v_3 & \cdots & -v_{n-1}\\
		0 & \bar v_2 & -v_3  & \cdots & -v_{n-1}\\
		0 & 0 & \bar v_3  & \cdots & -v_{n-1}\\
		\vdots & \vdots & \vdots & \ddots & \vdots\\
		0 & 0 & 0 & \cdots & \bar v_{n-1}
	\end{bmatrix}.
	\]
	Then we have
	\[
	V=\begin{bmatrix}
		0                       & \alpha \cr
		I_k\otimes \bar \cO    & \beta e
	\end{bmatrix}\in \R^{nk+1\times (n-1)k+1}, \quad
	V^TV=I, \, BV = 0.
	\]
\end{lemma}
\begin{proof}
	Denote the $j$-th column of $V$ by $V_j$ and define $J_s:=\{j^s_1,j^s_2,\dots, j^s_{n-1}\}$, where $j^s_r=(n-1)(s-1)+r$. Notice that $J_s$ is the index set of columns of $V$ in $s$-th block. $j\in J_{k+1}$ means $V_j$ is the last column of $V$.

	We first prove that $V^TV=I$, i.e., column vectors of $V$ is orthonormal. Let $i,j\in\{1,\dots,(n-1)k+1\}$. We consider the following cases: \\
	If $j\le(n-1)k$, then
	\[
	V_j^TV_j = jv_j^2 + \bar v_j^2 = \frac{j}{j+j^2} + \frac{j^2}{j+j^2} = 1.
	\]
	If $j=(n-1)k+1$, then
	\[
	V_j^TV_j = \alpha^2 + nk\beta^2 = (n^2+nk)\beta^2 = 1.
	\]
	Now let $i<j$.
	If $i,j\in J_s$ for some $s\le k$. Then,
	\begin{align*}
		V_i^TV_j &= iv_iv_j - \bar v_iv_j\\
		&= i\cdot\frac{1}{\sqrt{i+i^2}}\frac{1}{\sqrt{j+j^2}} - \frac{i}{\sqrt{i+i^2}}\frac{1}{\sqrt{j+j^2}} = 0.
	\end{align*}
	If $j=(n-1)k+1$. Then,
	\begin{align*}
		V_i^TV_j = -iv_i\beta + \bar v_i\beta = (-iv_i+iv_i)\beta =0.
	\end{align*}
	If $i\in J_s$, $j\in J_t$ with $s<t\le k$. For each row, at least one of the vectors has $0$ entry, so trivially $V_i^TV_j=0$. This proves the orthonormality.

	Secondly, we observe $BV=0$, i.e., $\range{V}\subseteq\nul(B)$. To this end, we will see that $BV_j=0$ for each $j=1,\dots,(n-1)k+1$. Fix $s\in\{1,\dots,k\}$. If $j=(n-1)k+1$,
	\[
	\begin{pmatrix}
		BV_j
	\end{pmatrix}_s = -\alpha + n\beta = -n\beta + n\beta=0, 
	\]
	Now assume that $j\le (n-1)k$. If $j\in J_s$, then
	\[
	\begin{pmatrix}
		BV_j
	\end{pmatrix}_s = -jv_j+\bar v_j = -jv_j+jv_j=0, \text{ for each }i=1,\dots,k.
	\]
	Otherwise, trivially $\begin{pmatrix}
		BV_j
	\end{pmatrix}_s=0$. This justifies $BV=0$.
\end{proof}

We leave open the question on how to exploit the structure of $V$ to
obtain efficient matrix-matrix multiplications of the form $VRV^T$
needed in our algorithm.

\subsubsection{DNN reformulation via facial reduction}
Recall that the lifting for $Y_x$ has the form
\[
Y_x = \begin{pmatrix} 1\cr x\end{pmatrix}
\begin{pmatrix} 1\cr x\end{pmatrix}^T,\quad
 x\in\{0,1\}^{kn}. 
\]
Hence, we can impose the redundant elementwise \LP relaxation bound constraint 
$0\leq Y\leq 1$. 
Moreover, the constraint $KY = 0$ is equivalent to applying
\FRp, i.e.,~we get
\[
Y\succeq 0, KY=0\iff Y=VRV^T, R\in\Sp^{nk+1-k}.
\]
We now observe a useful redundant trace constraint on 
$Y$ and transform it onto $R$. 
\begin{lemma}
Let $Y\in\Sknplusone, Y=VRV^T, R\in\Sc^{nk+1-k}$. 
Then
\[
 KY=0, \arrow(Y)=e_0 \implies
              \trace(Y) =  \trace(R) = k+1.
\]
\end{lemma}
\begin{proof}
	Recall that $K:=\matrixTwoOne{-e^T}{A^T}\matrixTwoOne{-e^T}{A^T}^T$. Since $\nul(K)=\nul\Bigg(\matrixTwoOne{-e^T}{A^T}^T\Bigg)$, 
	we have 
	\[0=KY \iff 0 =  \matrixThreeFour{-1}{e^T}{...}{0^T}{...}{...}{...}{...}{-1}{0^T}{...}{e^T}
					\matrixC{Y_{0,0}}{...}{Y_{0,nk}}{...}{...}{...}{Y_{nk,0}}{...}{Y_{nk,nk}}.\]
					
	By expanding the first column of the product, we get $\summ{i=1}{n}{Y_{jn+i,0}}=1, \forall j\in\{0,...,k-1\}$. Since $\arrow(Y)=e_0$, this implies that $\trace(Y)=Y_{0,0}+\summ{j=0}{k-1}{\summ{i=1}{n}{Y_{jn+i,0}}} = 1 + k$. 
Since we choose the facial vector $V$ to have orthonormal columns,
the facial constraint yields
\[
1+k=\trace(Y)=\trace(VRV^T)=\trace(RV^TV)=\trace(R),
\] 
\end{proof}
	
Next, we incorporate all these constraints into the \SDP
relaxation model to form the \DNN relaxation model. Define the two set
constraints
\begin{equation}
\label{eq:setYR}
\textdef{$\cY$} :=\{Y\in\Sc^{nk+1}: Y_{00}=1, \GG_{\cJ}(Y)=0,
\arrowz(Y)=0, 0\leq Y\leq 1\}, \,\,\textdef{$\fancy{R}$}:=\{R\in\Sp^{nk+1-k}: \trace(R)=k+1\}.
\end{equation}
Our \textdef{\DNN relaxation} model is:
	\begin{equation}
	\label{eq:DNNrelax}
	\text{(\DNN)}\qquad
	\begin{array}{cl}
		p^*_{\DNN}:=\min_{R,Y} &  \iprod{\hat{D}}{Y} \\
     		   \text{s.t.} &  Y=VRV^T  \\
				 &	Y\in \fancy{Y} \\
				 &	R\in \fancy{R}.
	\end{array}
	\end{equation}
Observe that every feasible $Y$ is both element-wise nonnegative and 
$\PSD$, i.e.,~this is a \DNN relaxation.
Moreover, by construction, \DNN is trivially a feasible problem and the
constraint $Y-VRV^T=0$ is trivially surjective. Therefore, we have a
regularized problem.

The splitting allows for the two cones to be handled separately.
Combining them into one and applying e.g.,~an interior point approach is
known to be
very costly. 
And, one cannot get high dual feasibility accuracy  
and therefore the approximate
optimal value we get is not a provable lower bound for~\Cref{prob:mainWB}.
In summary, the expense does not scale well with $N$,
and we cannot apply weak duality and use the dual solution as both
primal and dual feasibility are not highly accurate. We overcome this
problem for the splitting method in~\Cref{Lagdual} below.
\subsubsection{Characterization of optimality for \DNN relaxation}
\label{sec:optcond}
The generalized \KKT optimality conditions hold for~\cref{eq:DNNrelax}
with the normal cone $\NN_{\cY \times \cR}(Y,R)$. In addition,
the interior of the closed convex feasible set 
$\Int (\cY \times \cR) \neq \emptyset$ implies that
\[ \NN_{\cY \times \cR}(Y,R) = \NN_\cY(Y) \times \NN_\cR(R).
\]
We can now use the corresponding Lagrangian with dual variable $Z$:
\index{$\iota_S(\cdot)$, indicator function}
\index{indicator function, $\iota_S(\cdot)$}
\[
\cL(Y,R,Z) = \langle \hat D,Y\rangle + \langle Z,Y-VRV^T\rangle
	+\iota_{\YY}{(Y)}+\iota_{\RR}{(R)},
\]
where $\iota_S(\cdot)$ is the \emph{indicator function} for the set $S$.
Therefore the first-order optimality conditions to the 
problem in \cref{eq:DNNrelax} are:
a  primal-dual pair $(Y, R, Z)$ is optimal if, and only if,
		\begin{subequations}\label{eq:optcond}
			\begin{alignat}{5}
			Y & =   VRV^T,
             		 \quad R \in \RR, \,  Y \in \YY
            		& \quad \text{(primal feasibility)} \label{eq:DNNa}\\
			0&\in  -V^TZV+\NN_{\RR}(R)
            		&\quad \text{(dual $R$ feasibility)} \label{eq:DNNb} \\
			0&\in  \hat{D}+Z+\NN_{\YY}(Y)
            		&\quad \text{(dual $Y$ feasibility)} \label{eq:DNNc}
			\end{alignat}
		\end{subequations} 

By the definition of the normal cone, we can easily obtain the following 
		\Cref{thm:charactoptMAIN}.
		\begin{prop}[characterization of optimality for \DNN
in \cref{eq:DNNrelax}]
		\label{thm:charactoptMAIN}
		The primal-dual pair $(R,Y,Z)$ is optimal for \cref{eq:DNNrelax} if, and only if,
		\cref{eq:optcond} holds if, and only if,
		\begin{subequations}\label{eq:optcondnew}
			\begin{alignat}{4}
			R &=  \mathcal{P}_\RR (R+V^TZV) \label{eq:optcondnew_inR}  	\\
			Y& = \mathcal{P}_\YY (Y- \hat{D} - Z) \label{eq:optcondnew_inY}\\
			Y & =   VRV^T. \label{eq:optcondnew_pR}
			\end{alignat} 
		\end{subequations} 
		\end{prop}
\section{\sADMM algorithm; bounding; empirics}
\label{sect:ADMMalg}
The augmented Lagrangian corresponding to the \DNN relaxation
\cref{eq:DNNrelax}  with parameter $\beta>0$ is
\begin{equation}
	\label{eq:augmentLagrbeta}
	\LL_\beta(Y,R,Z):=\iprod{\hat{D}}{Y}+\iprod{Z}{Y-VRV^T}+\frac{\beta}{2}\normF{Y-VRV^T}^2+\iota_{\YY}{(Y)}+\iota_{\RR}{(R)}.
\end{equation}
 To solve our \DNN relaxation in  \cref{eq:DNNrelax},
we use the symmetric alternating directions method
 of multipliers $\sADMM$ that has intermediate updates of dual
multipliers $Z_t$: one dual update after the
$R$-update and then another after the $Y$-update. 
This approach has been used successfully in
\cite{BurkImWolk:20,LiPongWolk:19}. 
 Hence, both the $R$-update and the $Y$-update take into
account newly updated dual variable information.
(We include the details here for completeness.)

		 Let $Y_0\in \Sc^{nk+1}, Z_0\in \Sc^{nk+1}$.
The updates for all nonnegative integers $k\in \Zp$ are:
\index{nonnegative integers, $\Zp$}
\index{$\Zp$, nonnegative integers}
\begin{equation}
\label{e:admmup4}
\begin{array}{rcl}
R_{k+1}&=&\argmin_{R\in\Sc^{nk+1-k}}\LL_\beta(Y_k, R, Z_k)
\\
Z_{k+\frac{1}{2}}&=&Z_k+\beta(Y_k-VR_{k+1}V^T)
\\
Y_{k+1}&=&\argmin_{Y\in\Sc^{nk+1}}\LL_\beta( Y, R_{k+1},Z_{k+\frac{1}{2}})
\\
Z_{k+1}&=&Z_{k+\frac{1}{2}}+\beta(Y_{k+1}-VR_{k+1}V^T).
\end{array}
\end{equation}
	
	In our \DNN model \cref{eq:DNNrelax}, the objective function
is continuous and the feasible set is compact. By the Weierstrass 
theorem, an optimal primal pair $(Y^*, R^*)$ always exists. As seen
above, the
constraint is linear and surjective and strong duality holds for the
generalized \KKT conditions. (See the
optimality conditions in~\Cref{sec:optcond}). In fact, in our application
we modify the dual multiplier update using a projection, see
\Cref{lem:ZA} and \Cref{alg:sADMM}.
	
\subsubsection*{Explicit Primal updates for $R,Y$}
The success of our splitting method is dependent on efficiently solving
the subproblems.
 We start with using a spectral decomposition, implicitly
defined below, to get the:
\[\begin{array}{rcll}
R-\text{update} 
&=& \argmin_{R\in\Sc^{nk+1-k}}\LL_\beta(R, Y_k, Z_k) \\
&=& \argmin_{R\in\RR}\normF{Y_k-VRV^T+\frac{1}{\beta}Z_k}^2,  	&\text{by completing the square}\\
&=& \argmin_{R\in\RR}\normF{V^TY_kV-R+\frac{1}{\beta}V^TZ_kV}^2, &\text{since }V^TV=I\\
&=& \argmin_{R\in\RR}\normF{R - V^T(Y_k+\frac{1}{\beta}Z_k)V}^2	\\
&=& \PP_\RR[V^T(Y_k+\frac{1}{\beta}Z_k)V]					
\\ &=& U\Diag[\PP_{\simplex{k+1}}(\lambda)]U^T, & \text{spectral decomposition}		
\end{array}
\] 
where the $U\Diag(\lambda)U^T$ provides the spectral decomposition of $V^T(Y_k+\frac{1}{\beta}Z_k)V$ and then $\PP_{\simplex{k+1}}$ denotes the projection onto the
\textdef{simplex} $\simplex{k+1}:=\{x\in\Rnp:\iprod{e}{x}=1+k\}$, see
e.g.,~\cite{chen2011projection}.
	
		Next for the 
		\[\begin{array}{rcll}
Y-\text{update}
		&=&\argmin_{Y\in\Sc^{nk+1}}\LL_\beta(R_{k+1}, Y, Z_{k+\frac{1}{2}})
		\\
			&=& \argmin_{Y\in\YY}\normF{Y-[VR_{k+1}V^T - \frac{1}{\beta}(\hat{D}+Z_{k+\frac{1}{2}})]}^2   	&\text{by completing the square}\\
			&=& \PP_\YY\left(
VR_{k+1}V^T - \frac{1}{\beta}(\hat{D}+Z_{k+\frac{1}{2}})\right)\\
			&=&
\PP_{\rm arrowbox}\left(\PP_{\nul\left(\GG_{\hat{\cJ}}\right)}[VR_{k+1}V^T -
\frac{1}{\beta}(\hat{D}+Z_{k+\frac{1}{2}})]\right),
\end{array}
\] 
where $\GG_{\hat{\cJ}}$ is the gangster constraint linear transformation and
\textdef{$\PP_{\rm arrowbox}$} projects onto the polyhedral set 
$\{Y\in\Sc^{nk+1}: Y_{ij}\in [0,1], \arrow(Y) = e_0\}$.

		\subsubsection*{Dual updates}
		 The correct choice of the Lagrange dual
multiplier $Z$ is important in the progress of the algorithm 
and in obtaining strong lower bounds. In
addition, if the set of dual multipliers for all iterations is compact,
then it indicates the stability of the primal problem. Lagrange
multipliers are used to replace constraints by adding the appropriate
expression into the Lagrangian function. If an optimal
$Z^*$ for \cref{eq:DNNrelax} is known in advance, then it makes sense
that we do not need the corresponding
primal feasibility constraint $Y=VRV^T$. Hence, following the idea of
exploiting redundant constraints, we now aim to identify certain properties of an optimal dual multiplier and impose that property at each iteration of our algorithm. 

		\begin{lem}
\label{lem:ZA}
Let
\[
\cZ_A:=\left\{Z\in\Sknplusone: (Z+\hat{D})_{i,i}=0,
(Z+\hat{D})_{0,i}=0, (Z+\hat{D})_{i,0}=0, i=1,...,nk\right\}.
\]
Let $(Y^*, R^*, Z^*)$ be an optimal primal-dual pair 
for the \DNN in \cref{eq:DNNrelax}. Then, $Z^*\in\cZ_A$.
\end{lem}
\begin{proof}
				The proof of this fact uses the dual $Y$ feasibility
		condition \cref{eq:DNNc} and a reformulation of the $Y$-feasible set.
		The details are in \cite[Thm 2.1]{QAP} and
\cite{BurkImWolk:20}. 
	\end{proof}	

In view of \Cref{lem:ZA} we propose the following modification
of the symmetric \ADMM algorithm, e.g.,~\cite{MR3231988}.
Our modification is in the way we update the multiplier.
		At every initial or intermediate update of the multiplier we project the dual variable onto $\cZ_A$, i.e:
\begin{itemize}
\item 
$Z_{j+\frac{1}{2}}:=Z_j+\gamma\beta \PP_{\cZ_A}(Y_j-VR_{j+1}V^T)$;
\item
$Z_{j+1}:=Z_{j+\frac{1}{2}}+\gamma\beta\PP_{\cZ_A}(Y_{j+1}-VR_{j+1}V^T)$.
\end{itemize}
Note that a convergence proof using the modified updates is given
in~\cite[Thm 3.2]{QAP}.
Therefore, in view of the \ADMM updates~\cref{e:admmup4}
we propose the following~\Cref{alg:sADMM} with modified $Z$ updates.
The $\gamma \in (0,1)$ is the dual steplength.
	\begin{algorithm}
		\caption{\sADMM, modified symmetric \ADMMp}
\label{alg:sADMM}
		\begin{algorithmic}
			\STATE Initialization: $j=0,Y_j=0\in S^{nk+1},
Z_j=\PP_{\cZ_A}(0), \beta=\max(\floor{\frac{nk+1}{k}}, 1), \gamma=0.9$
			\WHILE{termination criteria are not met}
				\STATE $R_{j+1} =
U\Diag[\PP_{\Delta_{j+1}}(d)]U^T$ where $U\Diag(d)U^T=\eig(V^T(Y_j+\frac{1}{\beta}Z_j)V)$
				\STATE $Z_{j+\frac{1}{2}} = Z_j + \gamma\beta \PP_{\cZ_A}(Y_j - VR_{j+1}V^T)$
				\STATE $Y_{j+1} = \PP_{\rm arrowbox}[\PP_{\nul(\cG_{\hat{J}})}(VR_{j+1}V^T - \frac{1}{\beta}(\hat{D}+Z_{j+\frac{1}{2}}))]$
				\STATE $Z_{j+1} = Z_{j+\frac{1}{2}} + \gamma\beta \PP_{\cZ_A}(Y_{j+1} - VR_{j+1}V^T)$
\STATE  $j = j+1$
			\ENDWHILE
		\end{algorithmic}
	\end{algorithm}
\begin{remark}
In passing, we point out that we could choose any $\gamma \in (0,1)$
and $\beta>0$. Theoretically this is all what we need. 
In our numerical experiments for~\Cref{alg:sADMM} we
used an adaptive $\beta$ based on the discussion in~\Cref{sect:adaptstep}.
\end{remark}

\subsection{Bounding and duality gaps}
Strong upper and lower bounds allow for early stopping conditions as
well as proving optimality. We now provide provable upper and lower
bounds to machine precision.

\subsubsection{Provable lower bound to NP-hard problem}
\label{Lagdual}
 The Lagrangian dual function $g:\Sc^{nk+1}\rightarrow\R$
to the \DNN model that we use is 
	\[\begin{array}{rcl}
g(Z)  &=&\min_{R\in\RR,Y\in\YY}\iprod{\hat{D}}{Y}+\iprod{Z}{Y-VRV^T}\\
			       			&=&		 \min_{Y\in \YY, R\in\RR}\iprod{\hat{D}+Z}{Y} - \iprod{Z}{VRV^T}\\
			       			&=&	        \min_{Y\in \YY} \iprod{\hat{D}+Z}{Y} + \min_{R\in\RR} (-\iprod{V^TZV}{R})\\
			      			&=&		\min_{Y\in \YY} \iprod{\hat{D}+Z}{Y} - \max_{R\in\RR} \iprod{V^TZV}{R}\\
			       			&=&		\min_{Y\in \YY} \iprod{\hat{D}+Z}{Y} - \max_{\|v\|^2=(k+1)} v^TV^TZVv\\
			       			&=&		\min_{Y\in \YY}\iprod{\hat{D}+Z}{Y} - (k+1)\lambda_{\max}(V^TZV).
	\end{array}\]
Hence, at iteration $j$, and applying weak duality,
a lower bound to the optimal value of the \DNN model~\cref{eq:DNNrelax} is 
\begin{equation}
\label{eq:weakdualityDNN}
\begin{array}{rcl}
p^*_{\DNN} 
& \geq & 
\max_Z g(Z) 
\\&\geq &  
\min_{Y\in \fancy{Y}}\iprod{\hat{D}+Z_j}{Y} 
         - (k+1)\lambda_{\max}(V^TZ_jV).
\end{array}
\end{equation}
Note that from the definition of $\cY$ in~\cref{eq:setYR}, this bound is
found from solving: an \LP with a 
simplex type feasible set; and an eigenvalue
problem.  Thus both values can be found accurately and efficiently.
Moreover, since \DNN is a relaxation, weak duality implies that
this lower bound is a \emph{provable lower bound} for the original
NP-hard~\Cref{prob:mainWB}.

\subsubsection{Upper bounds}
As for the upper bound, we consider two strategies for finding feasible solutions to the \BCQP in \cref{eq:BCQP}. 
The $0$-column approach is to take all but the first element of this $0$-th column $Y(1\!:\!\text{end},0)$ and compute its nearest feasible solution to \BCQPp.
It is equivalent to the greedy approach of using only the maximum weight index for each consecutive block of length $n$, see~\cite[Section 3.2.2]{BurkImWolk:20}.
	
Alternatively, we use the eigenvector of $Y$ corresponding to
the largest eigenvalue. The Perron-Frobenius Theorem implies this eigenvector is nonnegative, as $Y$ is nonnegative.
We then compute the nearest feasible solution to \BCQPp. 
It is again equivalent to the greedy approach but with using the eigenvector.
	
Then, we compare the objective values for both approaches and select the upper bound with smaller magnitude. 
The relative duality gap at the current iterate $j$ is defined to be $\frac{UB_j-LB_j}{|UB_j|+|LB_j|+1}$ where $UB_j$, $LB_j$ denote the current best upper, and lower bound, respectively.
	\subsection{Stopping criterion}
 By \Cref{thm:charactoptMAIN}, we can define the primal and dual
 residuals of the $\sADMM$ algorithm at iterate $j$ as follows:
	\begin{itemize}
		\item Primal residual $r_j:=\|Y_j-VR_jV^T\|$;
		\item Dual-$R$ residual $s^R_j:=\|R_j -
\PP_\RR\left(R_j+V^TZ_jV\right)\|$;
		\item Dual-$Y$ residual $s^Y_j:=\left\|Y_j -
\PP_\YY\left(Y_j-\hat{D}-Z_{j+\frac{1}{2}}\right)\right\|$.
	\end{itemize}
	We terminate the algorithm once one of the following conditions is satisfied:
	\begin{itemize}
		\item The maximum number of iterations
$(\mathrm{maxiter}):=10^4+k(nk+1)$ is reached;
		\item The relative duality gap is less or equal to
$\epsilon$, a given tolerance;
\item $\mathrm{KKTres}:=\max\{r_j, s^R_j, s^Y_j\}<\eta$, a given tolerance.
		\item Both the least upper bound and the greatest lower
bound have not changed for $\mathrm{boundCounterMax}:=200$ times (stalling).
	\end{itemize}
	\subsection{Heuristics for algorithm acceleration}
		\subsubsection{Adaptive step size}
\label{sect:adaptstep}
		 We apply the heuristic idea presented in 
\cite{ADMMBoyd}, namely we bound the gap between the primal and dual
residual norms within a factor of $\mu:=2$ as they converge to 0. This
guarantees that they converge to 0 at about the same rate and one
residual does not overshoot the other residual by too much. Since a large penalty $\beta$ prioritizes primal feasibility over dual feasibility and a small penalty $\beta$ prioritizes dual feasibility over primal feasibility, we scale $\beta$ by a factor of $\tau^{\mathrm{incr}}:=2$ if the primal residual overshoots the dual residual by a factor of $\mu$ and scale $\beta$ down by a factor of $\tau^{\mathrm{decr}}:=2$ if the dual residual overshoots the primal residual by a factor of $\mu$. Otherwise, we keep $\beta$ unchanged. Namely,
		\[\beta_{j+1}:=\begin{cases}
			\tau^{\mathrm{incr}}\beta_j, & \norm{r_j}_2 > \mu\norm{s_j}_2;\\
					\frac{\beta_j}{\tau^{\mathrm{decr}}},
					&\norm{s_j}_2 > \mu\norm{r_j}_2;\\
					\beta_j, & \text{otherwise}.
				  \end{cases}\]
		\subsubsection{Transformation and scaling}
In this subsection, we consider translating and scaling the objective
function, i.e.,~$\hat{D}$.
Define the orthogonal projection map $P_V := VV^T$. Then, 
		\index{$\hat{D}$ scaled}
		\begin{equation}
		\label{eq:shifteigLtrans}
		\begin{array}{rcl}
		\iprod{\hat{D}}{Y} &:=& \iprod{\hat{D}+\alpha I}{Y}  - (n+1)\alpha\\
					   &=& \iprod{\hat{D}+\alpha I}{P_V YP_V}  - (n+1) \alpha\\
					   &=& \iprod{(P_V\hat{D}P_V+\alpha I)}{Y}   - (n+1) \alpha.
		\end{array}
		\end{equation}
		Hence, when finding the optimal $Y$, 
		\[\begin{array}{ll}
		\iprod{\hat{D}}{Y} \text{ is minimized} &\iff \delta\iprod{\hat{D}}{Y}=\iprod{\delta(P_V\hat{D}P_V+\alpha I)}{Y}-(n+1)\delta\alpha \text{ is minimized}\\
							       		&\iff \iprod{\delta(P_V\hat{D}P_V+\alpha I)}{Y} \text{ is minimized}.
		\end{array}\]
		This lets us transform $\hat{D}$ into
$\delta(P_V\hat{D}P_V+\alpha I)$ without changing the optimal solutions.
		 Numerical experiments show that once we scale $\hat{D}$
by some $\delta>0$, the convergence becomes faster for the
aforementioned input data distributions, e.g.,~providing a normalization for the
objective function matrix.

		

	\subsection{Numerical tests}
\label{sect:numerics}
We now illustrate the efficiency of our algorithm on medium and large
scale randomly generated problems. We observe that our \sADMM approach
finds the \emph{exact} solution with relative duality gap $<$ 1e-13,
i.e., in machine precision, in almost all of the instances we tried.
Sometimes, our algorithm gets stuck at the relative duality gap
$\approx$ 1e-13 and runs until it reaches the max iteration. These
indeterminate cases appear as ourliers in the numerical tests. Further discussion about the nontrivial duality gap follows in~\Cref{sect:multoptgaps}.

We use \textsc{Matlab} version 2025a on (fastlinux in caption): Dell PowerEdge, Two Intel Xeon Gold 6244 8-core 3.6 GHz (Cascade Lake), 192 GB for the tests in~\Cref{sec:hardness} and~\Cref{sec:hiddenedim}.
In~\Cref{sec:sADMMtests}, we use \textsc{Matlab} version 2022a on two linux servers:
(i) fastlinux: greyling22 Dell R840 $4$ Intel Xeon Gold $6254$, 
with $3.10$ GHz, $72$ core and $384$ GB for
\Cref{table:testingoutputsadmm}; and (ii) (biglinux in caption):
Dell PowerEdge R6625, two AMD EPYC $9754$ $128$-core $2.25$ GHz, $1.5$ TB for \Cref{table:largeprobs}.

\subsubsection{$\NP$-hardness of \WBPp}\label{sec:hardness}

To illustrate the difficulty in solving the original $\NP$-hard problem,
we solve $\BCQP$ in \cref{eq:BCQP} using the commercial \textsc{Gurobi}
solver, though total enumeration could be more competitive at times.
We set $n=(2\!:\!1\!:\!8)$, $k=(2\!:\!1\!:\!8)$ for the tests, i.e.,
each of the $k$ sets has the
same $n$ number of points. And we take the average of three tests for each
problem size. For each test, a random \EDM $D\in\Sc_+^{nk}$ is
generated with embedding dimension $d=2$.
\begin{figure}[ht!]
\centering
\includegraphics[width = 0.5\textwidth]{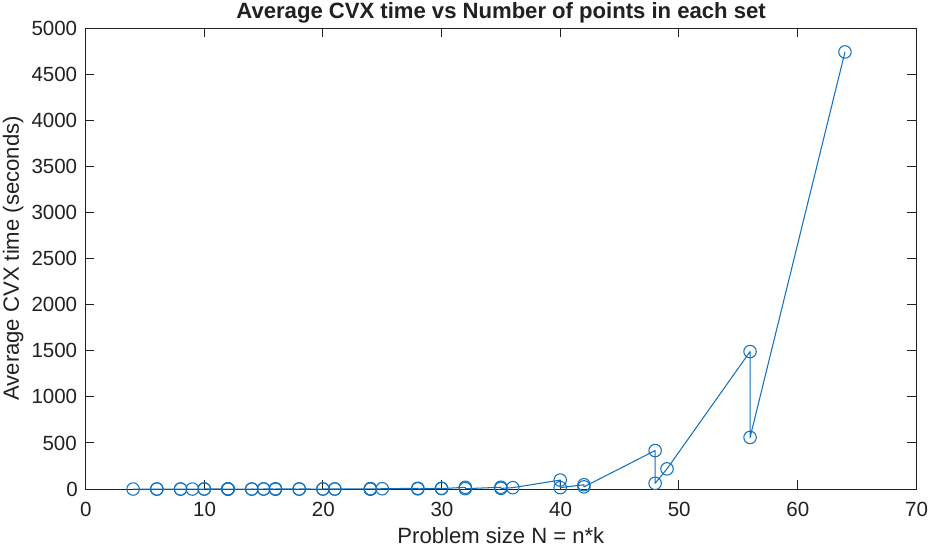}\hfill
\includegraphics[width = 0.5\textwidth]{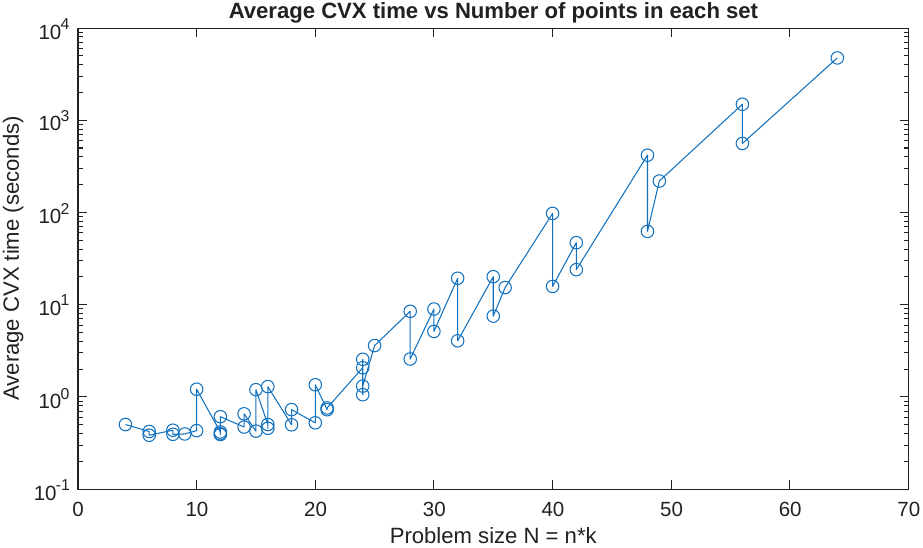}
\caption{ \textsc{Gurobi}: size $N=kn$ versus cpu time; illustrating exponential time}
\label{fig:BCQPGurobi}
\end{figure}
We can clearly see exponential growth in time
from~\Cref{fig:BCQPGurobi}. Note that the CVX time is related to the
number of feasible points, $n^k$. Hence, even if we have the same
problem size, e.g., $(n,k) = (7, 8), (8, 7)$, the number of vertices,
e.g., $n^k = 7^8, 8^7$, can be significantly different. Thus it appears
in the figures that we have two values for the same problem size.

In contrast, we see the slow (linear) growth for \sADMM, other than outliers,
for the computation time versus the size $N=\sum_in_i$, where $n_i$ are
the varying set sizes chosen randomly in
the interval of width $5$ about the given expected value $n$.
See~\Cref{fig:sizekntime}, page \pageref{fig:sizekntime}. The figure on
the right is log-log scale which reduces the effect of outliers, so we
can clearly see the linear relation. (Note that an outlier is 
generally a result of one of five instances having a positive duality gap.)
This was with $d = (3\!:\!3\!:\!6), k = (10\!:\!1\!:\!20), n =
(20\!:\!1\!:\!30)$, and each set has the same size. We run five problems for each data instance and take the average time.
\begin{figure}[ht!]
\centering
\includegraphics[width = 0.5\textwidth]{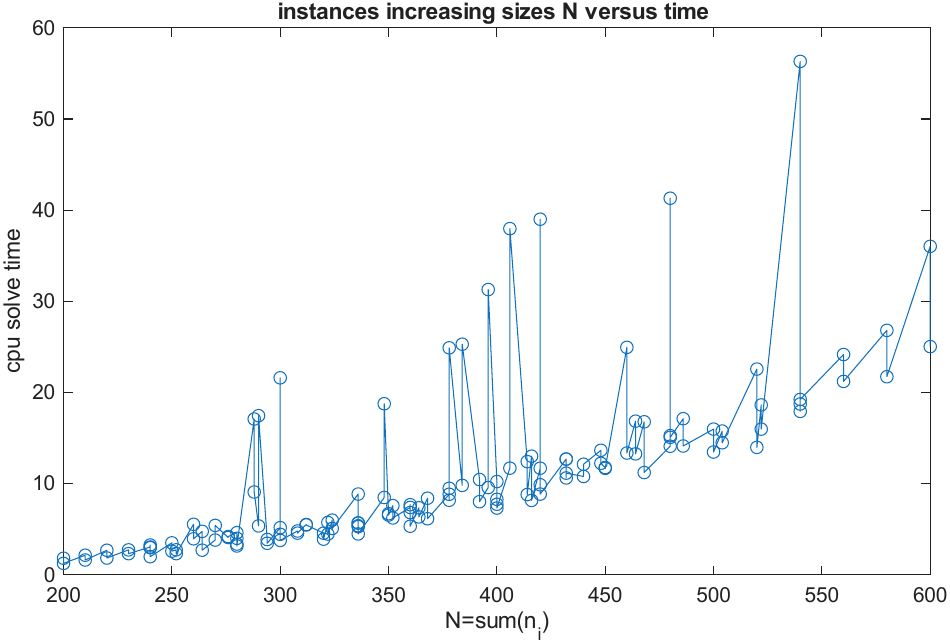}\hfill
\includegraphics[width = 0.5\textwidth]{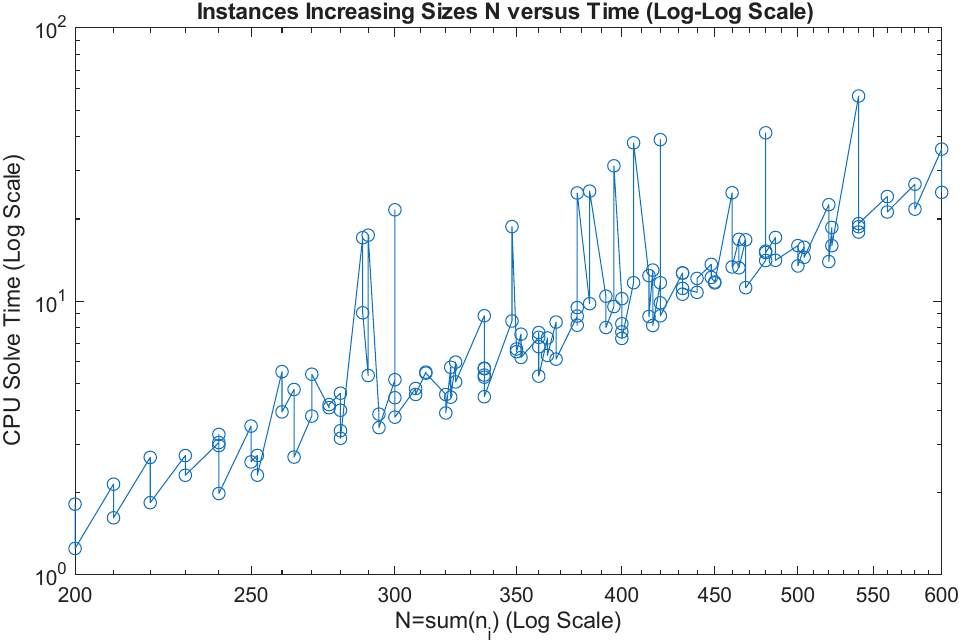}
\caption{\sADMM: size $N=kn$ versus cpu time; illustrating linear time}
\label{fig:sizekntime}
\end{figure}

\subsubsection{Success of \sADMM approach}\label{sec:sADMMtests}

The results below detail the efficiency and surprising success of our algorithm in finding the exact solution of the original NP-hard problem.

In \Cref{table:testingoutputsadmm}, page
\pageref{table:testingoutputsadmm}, we see a comparison between using
the \sADMM approach and CVX with the Sedumi solver. 
We can see the times for the CVXsolver increase dramatically.
The relative duality gap from CVX is \emph{not} a provable 
gap as the lower bound is
obtained using the dual optimal value minus the posted accuracy of the
solve from CVX;
we do not have accurate primal feasibility or dual feasibility from CVX.
Thus the relative gap is essentially the posted accuracy from CVX.
We do find a nearest feasible point to find the upper bound.
In summary, we see that the dramatic difference in time and the improved
accuracy with the guaranteed lower bound that verifies optimality.
\begin{table}[h]
\centering
\begin{tabular}{|ccc||cc|cc|} \hline
\multicolumn{3}{|c||}{dim/sets/size} & \multicolumn{2}{|c|}{Time (s)} & \multicolumn{2}{|c|}{rel. duality gap}\cr\hline  $d$&   $k$&   $N$&   ADMM & CVXsolver &    ADMM & CVXsolver  \cr\hline
    2 &     8 &    56 &   0.14 &   10.47 & 1.5e-14 &  1.3e-10  \cr\hline
    2 &     8 &    72 &   0.28 &   46.09 & 2.0e-14 &  9.1e-10  \cr\hline
    2 &     8 &    88 &   0.34 &   75.33 & 3.3e-15 &  5.2e-10  \cr\hline
    2 &     8 &   104 &   0.45 &  275.38 & 2.8e-14 &  6.1e-10  \cr\hline
    2 &     9 &    63 &   0.26 &   24.55 & 2.9e-15 &  3.2e-10  \cr\hline
    2 &     9 &    81 &   0.31 &   75.90 & -1.5e-16 &  3.5e-10  \cr\hline
    2 &     9 &    99 &   0.30 &  373.53 & 1.6e-14 &  3.8e-10  \cr\hline
    2 &     9 &   117 &   0.67 &  809.28 & -4.2e-14 &  5.9e-09  \cr\hline
    2 &    10 &    70 &   0.22 &   43.13 & 3.1e-14 &  1.2e-10  \cr\hline
    2 &    10 &    90 &   0.30 &  250.04 & 5.9e-16 &  3.0e-10  \cr\hline
    2 &    10 &   110 &   0.42 &  553.62 & 2.5e-14 &  4.4e-10  \cr\hline
    2 &    10 &   130 &   0.62 & 1555.54 & 3.4e-15 &  2.8e-09  \cr\hline
    3 &     8 &    56 &   0.10 &    9.55 & 4.6e-14 &  1.5e-10  \cr\hline
    3 &     8 &    72 &   0.38 &   46.60 & 2.6e-15 &  5.8e-10  \cr\hline
    3 &     8 &    88 &   0.21 &   77.41 & -3.1e-15 &  6.6e-10  \cr\hline
    3 &     8 &   104 &   0.38 &  280.12 & 1.7e-14 &  7.6e-10  \cr\hline
    3 &     9 &    63 &   0.15 &   19.97 & 3.4e-15 &  1.6e-10  \cr\hline
    3 &     9 &    81 &   0.13 &   81.98 & 7.7e-15 &  3.3e-10  \cr\hline
    3 &     9 &    99 &   0.24 &  360.76 & -2.8e-15 &  4.5e-10  \cr\hline
    3 &     9 &   117 &   0.62 &  803.72 & 4.7e-14 &  5.2e-09  \cr\hline
    3 &    10 &    70 &   0.21 &   39.18 & 5.2e-15 &  1.7e-10  \cr\hline
    3 &    10 &    90 &   0.23 &  236.75 & 9.1e-17 &  1.4e-10  \cr\hline
    3 &    10 &   110 &   0.29 &  562.02 & 3.0e-15 &  3.5e-10  \cr\hline
    3 &    10 &   130 &   0.76 & 1473.45 & 3.3e-14 &  4.7e-09  \cr\hline
\end{tabular}

\caption{Comparing ADMM with CVX Solver Sedumi}
\label{table:testingoutputsadmm}
\end{table}

We include large problems  in \Cref{table:largeprobs},
page~\pageref{table:largeprobs}. Each set has a constant number of elements $n$. Other than outliers, the times are very reasonable. 

\begin{table}[h]
\centering
\begin{tabular}{|ccc||c|c|} \hline
\multicolumn{3}{|c||}{dim/sets/size} & {Time (s)} & {rel.duality gap} \cr\hline
  $d$&   $k$&   $N$&    ADMM &    ADMM  \cr\hline
    8 &    30 &  1200 &  104.81 & -1.9e-15  \cr\hline
    8 &    30 &  1230 &   67.08 & -3.2e-14  \cr\hline
    8 &    31 &  1240 &   94.90 & -1.3e-14  \cr\hline
    8 &    31 &  1271 &   81.93 &  2.5e-14  \cr\hline
    8 &    32 &  1280 &   75.29 &  3.1e-14  \cr\hline
    8 &    32 &  1312 & 2025.42 &  1.8e-13  \cr\hline
    9 &    30 &  1200 & 4586.51 &  1.2e-13  \cr\hline
    9 &    30 &  1230 &   63.91 &  2.6e-14  \cr\hline
    9 &    31 &  1240 &   93.96 &  3.3e-14  \cr\hline
    9 &    31 &  1271 &   71.31 &  3.2e-14  \cr\hline
    9 &    32 &  1280 &   92.89 &  3.6e-14  \cr\hline
    9 &    32 &  1312 &   86.67 & -2.2e-13  \cr\hline
\end{tabular}

\caption{Large problems with \sADMM on biglinux server}
\label{table:largeprobs}
\end{table}

\subsubsection{Hidden embedding dimension}\label{sec:hiddenedim}
The embedding dimension $d$ is \emph{hidden} in our model as we only
use the distances between the $N$ points. This is why we do not see $d$ 
in \Cref{fig:sizekntime}.
The size of the \DNN model is $N+1$ irrespective of $d$.
The \emph{hardness} of the problem is often hiding in the rank of the 
\emph{optimal solution} $Y$ of the \DNN relaxation, 
i.e.,~if the rank is one, then we
have solved the original NP-hard problem. However, if the rank 
of the optimal $Y$ is large,
then the heuristics for the upper and lower bounds may not be enough to
find an optimal solution for the original NP-hard~\Cref{prob:mainWB}. 
However, the rank of the objective matrix $D$ that comes from $d$ did
not influence this at all.
In the various tests that we have done we did not see any changes in the
solution times or the ability to obtain a near zero duality gap for
wide varying values of $d$.

\begin{figure}[ht!]
\centering
\includegraphics[width=.5\textwidth]{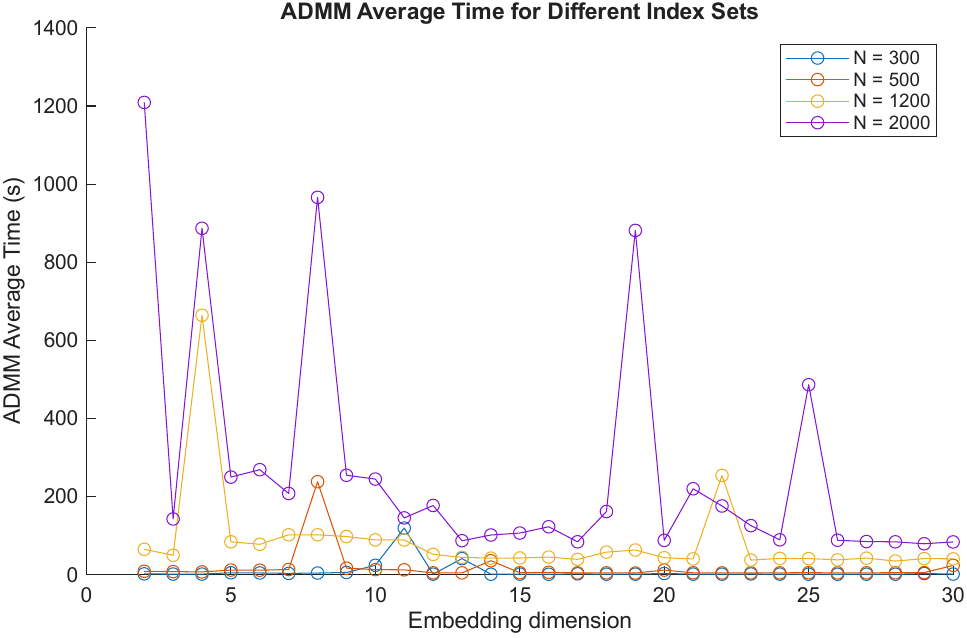}\hfill
\includegraphics[width = 0.5\textwidth]{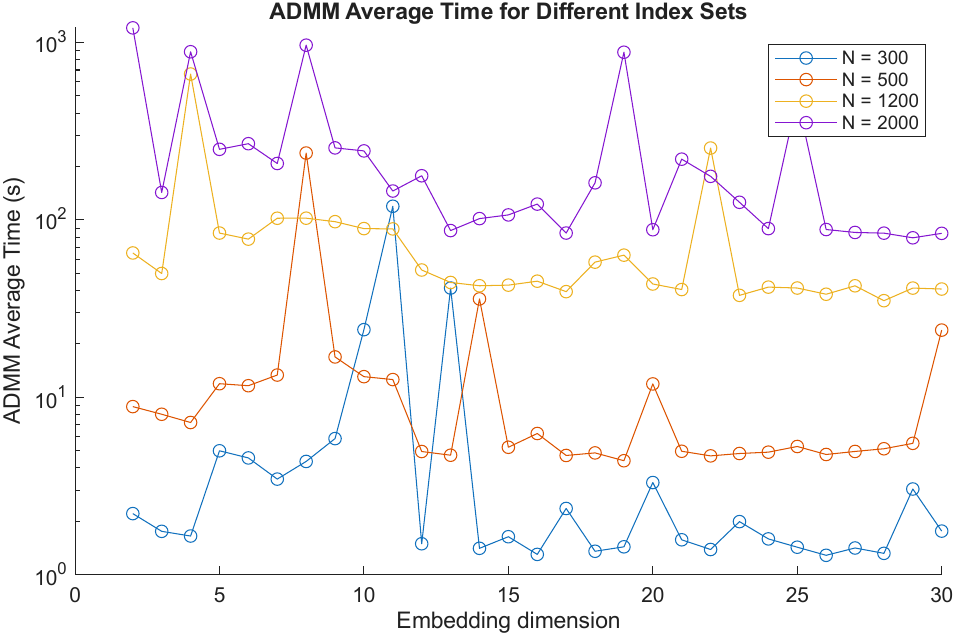}
\caption{\sADMM: Embedding dimension $d$ versus cpu time in various sizes}
\label{fig:diffedim}
\end{figure}

\Cref{fig:diffedim} in~\pageref{fig:diffedim} shows cpu time for our
$\sADMM$ method when we fix the problem size $N=nk$ and vary the
embedding dimension $d$. Random problems are generated with $(n,k)
\in\{10,40\}\times\{30,50\}$, and $d\in\{2,\dots,30\}$.
The cpu time is dependent only on problem size and not on the
embedding dimension.

\section{Multiple optimal solutions and duality gaps}
\label{sect:multoptgaps}
We now show that \emph{multiple optimal} solutions for the original hard problem
can lead to a duality gap between the optimal value of the
original NP-hard problem and the lower bound found from the \DNN relaxation.

\subsection{Criteria for duality gaps}
To find duality gaps for \SDP relaxations, we want to find optimal
points for the relaxation that are not vertices, i.e.,~they are
outside of the polyhedral set formed from
the convex hull of the lifted vertices. The key for this is having
multiple optimal solutions for the original problem. The following 
\Cref{lem:genbubblegap,cor:dualgapunc} provides a construction for
obtaining a positive gap between the optimal values
of a \emph{general} hard problem with multiple optimal solutions and its \DNN
relaxation.

\begin{lemma}
\label{lem:genbubblegap}
Let $\left\{x_i\right\}_{i=1}^n\subset \Rnp$ 
be a linearly independent set.
Define the lifted vertices and barycenter, respectively,
\[
\left\{X_i= x_ix_i^T\right\}_{i=1}^n \subset \Sn, \quad
\hat X:= \frac{1}{n}\sum_{i=1}^nX_i.
\]
Then $\hat x:=\sum_i x_i > 0$, and
\[
\hat X\in \Snpp \cap \Rnpp \quad (=\Int \DNN).
\]
\end{lemma}
\begin{proof}
That $\hat x  > 0$ follows by contradiction.
That $\hat X\in \Rnpp$ now follows.
Now note that $X_i\succeq 0, \forall i$, and so $\hat X\succeq 0$ as well. 
Now suppose that $0=\hat X v$, for some 
$0\neq v\in \Rn$. Then 
\[
0 = v^T\hat Xv = v^T\sum_iX_i v \implies 0=v^TX_iv, \forall i \implies
(v^Tx_i)^2 = 0, \forall i \implies v = 0,
\]
by the linear independence assumption; thus contradicting $v\neq 0$ and
yielding $\hat X\succ 0$.
\end{proof}

\begin{cor}
\label{cor:dualgapunc}
Suppose that the hypotheses of~\Cref{lem:genbubblegap} hold.
Let the points $x_i, i=1,\ldots, n$ be given (multiple) optima for a given
hard minimization problem, i.e.,
\[
(P) \quad  p^* = \min\left\{ x^TQx \,:\,   x\in \{0,1\}^n\right\} =
 x_i^TQx_i, \, i=1,\ldots,n.
\]
Moreover, suppose that there exists a feasible $y$ that is not optimal for (P), 
\[
y\in \{0,1\}^n, y^TQy > p^*. 
\]
Then the \DNN relaxation has feasible points $Y=yy^T,Z$ such that 
\[
\trace YQ > p^* > \trace ZQ,
\]
i.e.,~$Z$ yields a duality gap.
\end{cor}
\begin{proof}
From~\Cref{lem:genbubblegap} we have that the barycenter
$\hat X$ of the $\{x_i\}$ satisfies $\hat X\in \Int \DNN$.
Note that $\trace YQ = y^TQy > p^* = \trace \hat XQ$. Therefore, 
$\trace (\hat X - Y)Q<0$, and for $\epsilon > 0$,
\[
\trace (\hat X + \epsilon (\hat X-Y)) Q 
= p^* + \epsilon\trace (\hat X-Y)Q < p^*.
\]
Moreover, the line segment $[Y,\hat X + \epsilon (\hat X-Y)]$ is
feasible for the \SDP relaxation for small enough $\epsilon > 0$ by
$\hat X \in \Int \DNNp$.
Therefore, we set $Z_\epsilon= \hat X + \epsilon (\hat X-Y), 0<\epsilon <<1$ 
and obtain a duality gap for any such $Z=Z_\epsilon$.
\end{proof}

We can extend this theory to problems with general linear
constraints $Ax=b$ by using \FRp. We now specifically extend it to 
our \BCQP in \cref{eq:BCQP}. After \FR, we need $nk+1-k$ linearly independent
optimal points. This can be obtained when we choose $k >> n$.
Recall the matrix $K$ in \cref{eq:matK} used for facial reduction and the 
facially reduced \DNN relaxation in~\cref{eq:DNNrelax}.

\begin{cor}
\label{cor:dualgapWass}
We consider the \BCQP in \cref{eq:BCQP} with optimal value $p^*$,
and the \DNN relaxation in~\cref{eq:DNNrelax}. Let 
\[
\left\{ y_i=\begin{pmatrix}1\cr x_i\end{pmatrix}
\right\}_{i=1}^{nk+1-k} \subset \R_+^{nk+1}
\] 
be a linearly independent set that are optimal for \BCQP
and with $\sum_i y_i>0$.
Define the lifted vertices and barycenter, respectively,
\[
\left\{Y_i= y_iy_i^T\right\}_i, \forall i, \quad
\hat Y:= \frac{1}{nk+1-k}\sum_{i=1}^{nk+1-k}Y_i.
\]
Moreover, suppose that there exists a feasible $\bar x$ for \BCQP that
is not optimal. Then 
\[
\hat Y = V\hat RV^T \succeq 0, \, \hat Y > 0,  \, \hat R \succ 0.
\]
And there exists $Z=VR_ZV^T, R_Z\succ 0$ with optimal value $\trace DZ <
p^*$, yielding a duality gap.
\end{cor}
\begin{proof}
First note that incident vectors are feasible for the linear 
constraints and this guarantees that we have enough
feasible points to guarantee that the barycenter satisfies $\hat Y>0$.
All lifted feasible points of the relaxation are in  the minimal face
and have a corresponding matrix $R\in\mathbb{S}^{nk+1-k}$ for the facial reduction $Y=VRV^T$.
Since $R\succ 0$ after the \FRp, we can apply the same proof as
in~\Cref{cor:dualgapunc}. In addition, note that the linear constraints,
the arrow constraint and gangster constraints, remain satisfied in the
line formed from any two feasible points.
\end{proof}

\subsection{Wheel of wheels examples with a duality gap}
We illustrate the above theory with some specific problems
with special structure that have multiple optimal
solutions for the original $\NP$-hard problem. We see that a duality
gap can exist between the optimal value of the original $\NP$-hard problem and
the optimal value of the \DNN relaxation.

\begin{example}[Wheel of wheels]
		 We next present another input data distribution for
which the duality gap between the optimal value of the \BCQP formulation
and the Lagrangian dual value is non-trivial. The issue is again the
non-uniqueness of the optimal solutions and the $\sADMM$ algorithm fails to break ties among them.
		
The data distributions compose of a wheel of wheels,
i.e.,~a wheel with an odd
number of sets each of which is a wheel. Hence we call it an odd wheel. Given problem size parameters $(k, n, d)$, define 
		\begin{itemize}
			\item $\theta_k:=\frac{2\pi}{k}$. 
			\item a set of $k$ centroids encoded by a matrix $C\in\R^{k\times 2}$ such that 
				\[C(i,:) = \matrixOneTwo{\cos(i\theta_k-\theta_k)}{\sin(i\theta_k-\theta_k)}, i=1,...,k.\] 
			\item the radius of each cluster $r_k:=\frac{\sqrt{\cos(\theta_k-1)^2+\sin\theta_k^2}}{4}$. 
			\item the set of input points encoded by a matrix $P:=(C\otimes e_k) + r_k(e_k\otimes C)\in\R^{k^2\times 2}$.
		\end{itemize}
		
		When $k$ is odd, there exists more than one optimal
solution. A simple example with $k=3=n$ follows 
in~\Cref{fig:ce1}, page~\pageref{fig:ce1}.
We use the corresponding nine points in the configuration matrix ordered
$1-9$ counter-clockwise in the triangles ordered counter-clockwise.
 \[
P=
\left[
\begin{array}{cc} 1.7536 & 0.0137\\ 0.6195 & 0.6362\\ 0.6239 & -0.6643\\ 0.2590 & 0.8609\\ -0.8839 & 1.5449\\ -0.8692 & 0.2201\\ 0.2629 & -0.8740\\ -0.8937 & -0.2100\\ -0.8721 & -1.5275 \end{array}
\right]
\]
The distances are ordered by choosing the points in lexicographic order:
\[
(1,1), (1,2), \ldots  (3,3).
\]
The unique minimum distance is $11.1607$ obtained from the points
$2,3,2$ and with the primal optimal $x^* = \begin{pmatrix} 
0 & 1 & 0 & 0 &0 & 1  & 0 & 1 & 0 \end{pmatrix}^T$.
The optimal value from CVX to $9$ decimals precision is $10.8246$, thus
verifying an empirical duality gap of $.3$ to $9$ decimals precision.
The motivation for this counter-example is to have near optimal
solutions. One could shrink triangle two to make point $4$ equidistant
to points $7,8$ and move point $2$ closer to point $3$ 
and thus have a tie optimal solution.\footnote{The wheel graph was used
successfully to obtain duality gaps for the second lifting of the
max-cut problem, see~\cite{AnWo:00}.} Note that the maximum distance is $56.0227$ obtained from points $(1,2,3)$.

\begin{figure}[h]
    \centering
	\begin{tikzpicture}

	\tikzset{
		dot/.style={fill,circle,inner sep=1.5pt, color=OliveGreen},
		x_mark/.style={red,cross out,draw,thin,inner sep=1.5pt},
		red_circle/.style={draw=red,circle,thick,inner sep=2.5pt}
		}

	\begin{axis}[
		axis lines=middle,
		xmin=-2.0, xmax=2.0,
		ymin=-2.0, ymax=2.0,
		xtick={-2, -1.5, -1, -0.5, 0, 0.5, 1, 1.5, 2},
		ytick={-2, -1.5, -1, -0.5, 0, 0.5, 1, 1.5, 2},
		grid=major,
		grid style={thin,gray!20},
		xticklabel style={font=\small},
		yticklabel style={font=\small},
		width=10cm,
		height=10cm,
		enlarge x limits=false,
		enlarge y limits=false
	]

	\draw[blue,thin] (1.7536, 0.0137) -- (0.6195, 0.6362) -- (0.6239, -0.6643) -- cycle;
	\node[dot,label={[node distance=1pt]above:(1,1)}] at (1.7536, 0.0137) {};
	\node[dot,label={[node distance=1pt]above right:(1,2)}] at (0.6195, 0.6362) {};
	\node[red_circle] at (0.6195, 0.6362) {};
	\node[dot,label={[node distance=1pt]below:(1,3)}] at (0.6239, -0.6643) {};
	\node[x_mark] at (0.999, -0.0048) {};

	\draw[blue,thin] (0.2590, 0.8609) -- (-0.8839, 1.5449) -- (-0.8692, 0.2201) -- cycle;
	\node[dot,label={[node distance=1pt]above right:(2,1)}] at (0.2590, 0.8609) {};
	\node[dot,label={[node distance=1pt]above:(2,2)}] at (-0.8839, 1.5449) {};
	\node[dot,label={[node distance=1pt]left:(2,3)}] at (-0.8692, 0.2201) {};
	\node[red_circle] at (-0.8692, 0.2201) {};
	\node[x_mark] at (-0.498, 0.8753) {}; 

	\draw[blue,thin] (0.2629, -0.8740) -- (-0.8937, -0.2100) -- (-0.8721, -1.5275) -- cycle;
	\node[dot,label={[node distance=1pt]below:(3,1)}] at (0.2629, -0.8740) {};
	\node[dot,label={[node distance=1pt]below left:(3,2)}] at (-0.8937, -0.2100) {};
	\node[red_circle] at (-0.8937, -0.2100) {};
	\node[dot,label={[node distance=1pt]below:(3,3)}] at (-0.8721, -1.5275) {};
	\node[x_mark] at (-0.5009, -0.8705) {};

	\end{axis}
\end{tikzpicture}
    \caption{Duality gap for wheel of wheels: k=3=n}
    \label{fig:ce1}
\end{figure}
However, when $k$ is even, only one optimal solution
clearly exists and the duality gap becomes trivial. An example with $k=6=n$ 
follows in~\Cref{fig:ce2}, page~\pageref{fig:ce2}.
		\begin{figure}[ht!]
			\centering
			\includegraphics[width=15cm]{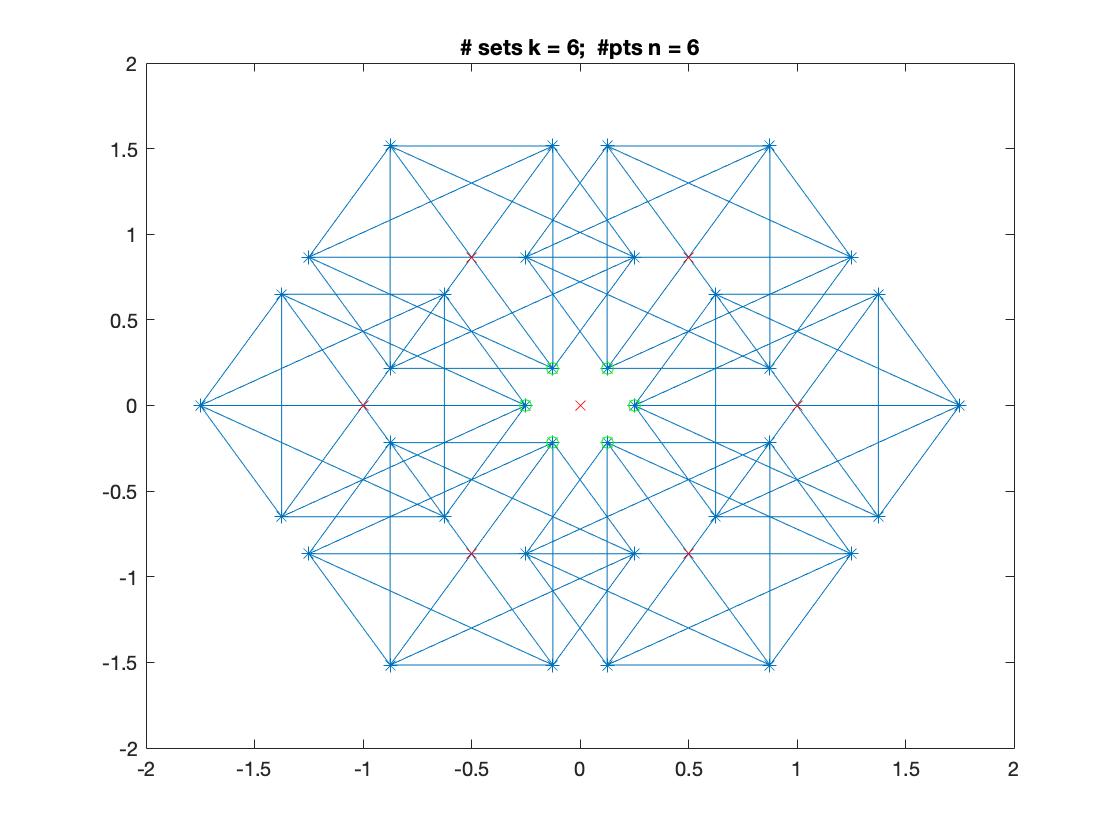}
			\caption{No duality gap for wheel of wheels: k=6=n} 
			\label{fig:ce2}
		\end{figure}
\end{example}

	\section{Conclusion}
\label{sect:concl}
 In this paper we presented a strategy for solving a class
of $\NP$-hard binary
quadratic problems, the simple Wasserstein barycenter problem. 
This involves formulating a \DNN relaxation, 
applying \FR that gives rise to a natural splitting for a 
symmetric alternating directions method of multipliers 
 \sADMM with intermediate update of multipliers
and strong upper and lower bounding techniques. In particular, the
structure of both the primal and dual solutions is exploited in the
updating steps of the \sADMMp.
We applied this to
the $\NP$-hard computational
problem called the simplified Wasserstein barycenter problem.

Surprisingly, for the random problems we generated the gap between
bounds was zero to machine precision
and we were able to provably solve the original $\NP$-hard
optimization problem, i.e.,~if there were another optimal solution, then
it would yield the same optimal value to machine precision.
This coincided with $\rank(Y^*) = 1$ for the optimal solution
found for the \DNN relaxation. We observed that the embedding dimension
$d$ is hidden in the \DNN relaxation.
However, for specially constructed input
data that has near multiple optimal solutions, the algorithm had difficulty
breaking ties and resulted in nontrivial gaps between lower and upper
bounds coinciding with $\rank(Y^*)>1$,
i.e.,~the original Wasserstein problem was not solved to
optimality.

	As for future research, we want to better understand the theoretical
reasons for the positive duality gaps and find more classes of problems
where this occurs. Does the lack of gaps correspond to large volumes for
the normal cones at points on the boundary of the feasible set?
In addition, we want to understand what happens under
small perturbations to problems with duality gaps, i.e.,~if the gaps can
be closed with perturbations.

Finally, we are gathering data about airports in North America by state
and province in
URL:
\href{https://ourairports.com/continents/NA/}
{ourairports.com/continents/NA/}. 
We plan on solving the problem of
finding the best hub in each state (or province) in order to find the
location for the best hub for the country.

{\bf Acknowledgment} The authors would like to thank Jiahui (Jeffrey) Cheng for
his contributions to the work in this paper. This joint research
continued from the work jointly started that resulted
in Mr Cheng's Master's thesis~\cite{JeffCheng:23}. 
The content and codes in this paper have changed.
The authors would also like to thank
Walaa M. Moursi for her contributions as a joint supervisor for Jeffrey
Cheng and for her work in an early version of this paper.

\cleardoublepage
\phantomsection
 \addcontentsline{toc}{section}{Index}
 \printindex
 \label{ind:index}

\cleardoublepage
\phantomsection
\addcontentsline{toc}{section}{Bibliography}
\bibliographystyle{siamplain}
\bibliography{.master,.edm,.psd,.bjorBOOK,jeff}
\label{bib:bibl}

\end{document}